\documentclass{amsart}
\usepackage{amssymb,amsmath,mathrsfs,amsfonts}
\usepackage{amscd, amssymb, amsmath, amsthm, graphics,graphicx}
\usepackage{amsmath,amsfonts,amsthm,amssymb}
\usepackage{mathabx}
\usepackage{fancyhdr}
\usepackage[all]{xy}
\usepackage[colorlinks=true]{hyperref}
\usepackage{multirow}
\usepackage{enumitem}
\numberwithin{equation}{section}

\newtheorem{theorem}{Theorem}[section]
\newtheorem{lemma}[theorem]{Lemma}
\newtheorem{proposition}[theorem]{Proposition}
\newtheorem{corollary}[theorem]{Corollary}

\theoremstyle{definition}
\newtheorem{definition}[theorem]{Definition}

\newtheorem{remark}[theorem]{Remark}

\begin{document}

\title[Two aspects of graph $3$-manifold groups]{Two aspects of graph $3$-manifold groups}

\author{Hongbin Sun}
\address{Department of Mathematics, Rutgers University - New Brunswick, Hill Center, Busch Campus, Piscataway, NJ 08854, USA}
\email{hongbin.sun@rutgers.edu}

\thanks{The author is partially supported by the Simons Collaboration Grant 615229.}

\date{\today}
\begin{abstract} 
We prove that fundamental groups of graph $3$-manifolds are virtually poly-free and lie in the family Lex. As a consequence, we prove that all finitely generated $3$-manifold groups also have these two properties. The first property is a purely group-theoretical concept, and the second is related to the left-exactness property of bounded cohomology of groups. Both properties are proved by constructing sequences of covers of graph $3$-manifolds.
\end{abstract}

\maketitle
\vspace{-.5cm}
\section{Introduction}

Our understanding of $3$-manifold groups has expanded enormously in the past two decades. The peak is Agol's virtually compact specialization theorem (\cite{Agol}) of hyperbolic $3$-manifold groups, which is based on Wise's machinery on cube complexes \cite{Wise}. 

After Agol's work, more $3$-manifold groups are proved to be virtually special, including: groups of graph $3$-manifolds with non-positively curved metrics (\cite{Liu}, c.f. \cite{PW1}), and mixed $3$-manifold groups (\cite{PW2}). We do not define virtually special groups here, but want to point out the following significant consequences: these $3$-manifold groups virtually (have finite-index subgroups that) are subgroups of right-angled Artin groups, virtually have surjections to non-abelian free groups, virtually are free-by-cyclic or surface-by-cyclic groups, and are linear over $\mathbb{Z}$, etc (see \cite[(H.5),  (H.13), (H.20), (H.31)]{AFW}, respectively). For more applications of Agol's virtually compact specialization theorem, see the survey papers \cite{AFW} and \cite{LS}.

In some sense, currently the least understood $3$-manifold groups are those of closed graph $3$-manifolds without non-positively curved metrics, since their groups are not virtually special (\cite{Liu}) and the above machinery cannot be applied to them. For example, essentially the only $3$-manifold groups that are unknown to be linear are groups of these graph $3$-manifolds (\cite[Problem 3.37]{BKR}). Although graph $3$-manifold groups virtually are graph-of-groups with $\text{free}\times \mathbb{Z}$ vertex groups and $\mathbb{Z}^2$ edge groups, the pasting pattern makes graph $3$-manifold groups hard to understand, and the level of difficulty depends on the property one wants to investigate.

The starting point of this paper is the following application of the virtual specialization results.
\begin{theorem}\label{virtualfiber}
Let $M$ be a compact, orientable, irreducible $3$-manifold with empty or tori boundary and infinite fundamental group, and we assume that $M$ admits a non-positively curved metric if it is a closed graph $3$-manifold. Then $M$ has a finite cover $\tilde{M}$ that is either a surface bundle over the circle, or a circle bundle over a surface.
\end{theorem}

If $M$ is a hyperbolic $3$-manifold, a mixed $3$-manifold, or a graph $3$-manifold with a non-positively curved metric, Theorem \ref{virtualfiber} follows from \cite[(H.20)]{AFW}, and these manifolds are virtually surface bundles over the circle. The remaining $3$-manifolds either have solvable fundamental groups, or are Seifert $3$-manifolds. They are finitely covered by torus bundles over the circle, or circle bundles over surfaces, respectively. 

For the $\tilde{M}$ in Theorem \ref{virtualfiber}, we have an exact sequence $$1\to A\to \pi_1(\tilde{M})\to B\to 1,$$ where either $A\cong \mathbb{Z}$ and $B$ is a surface/free group, or $A$ is a surface/free group and $B\cong \mathbb{Z}$. 

In this paper, by free groups, we mean free groups generated by zero, finitely many, or infinitely many elements. By surface groups, we mean fundamental groups of closed surfaces with non-positive Euler characteristics, including $\mathbb{Z}^2$ and $\pi_1$ of the Klein bottle.

Theorem \ref{virtualfiber} implies that the $\pi_1(M)$ in this result is virtually poly-free (c.f. \cite[(H.30)]{AFW}) and lies in the family Lex (which is well-known and will be proved in Section \ref{morelexsection}). In this paper, we will study these two properties for graph $3$-manifold groups.

A group $G$ is said to be {\it poly-free} if there is a finite sequence of subgroups 
$$G=G_0\vartriangleright G_1\vartriangleright G_2 \vartriangleright\cdots \vartriangleright G_n=\{1\},$$ such that $G_i/G_{i+1}$ is isomorphic to a (possibly infinitely generated) free group for all $i=0,1,\cdots,n-1$. Since surface groups are free-by-cyclic, where the free group might be infinitely generated, the $\pi_1(\tilde{M})$ in Theorem \ref{virtualfiber} is poly-free, with $n\leq 3$. So the $\pi_1(M)$ in Theorem \ref{virtualfiber} has a poly-free finite-index subgroup, thus is virtually poly-free.

Poly-free groups have hierarchical structures built from free groups, which are fundamental objects in combinatorial and geometric group theory. So (virtually) poly-freeness is a nice property with many consequences. Many groups related to low-dimensional topology are poly-free, including: surface groups, groups of $3$-dimensional bundles, pure braid groups, and right-angled Artin groups (\cite{HS}). 

For the graph $3$-manifolds not covered by Theorem \ref{virtualfiber}, we will prove their groups are also virtually poly-free.
\begin{theorem}\label{polyfree}
Let $M$ be a compact graph $3$-manifold, then $\pi_1(M)$ is virtually poly-free.
\end{theorem}

Theorem \ref{polyfree} implies that all finitely generated $3$-manifold groups are virtually poly-free.
\begin{corollary}\label{morepolyfree}
Let $G$ be a finitely generated $3$-manifold group, then $G$ is virtually poly-free.
\end{corollary}

A group lies in the family Lex if it satisfies a left-exactness property on bounded cohomology of groups (defined in \cite{Bou}), see Definition \ref{lexdef}. Some properties of Lex groups are summarized in Proposition \ref{lexproperty}, which implies that groups in Theorem \ref{virtualfiber} lie in the family Lex (see Section \ref{morelexsection}). It is unknown whether there exists a group that does not lie in the family Lex, and whether the free product of two Lex groups still lies in the family Lex.

In this paper, we prove that all non-virtually fibered graph $3$-manifold groups lie in the family Lex.

\begin{theorem}\label{lex}
Let $M$ be a closed graph $3$-manifold that is not virtually fibered, then $\pi_1(M)$ lies in the family Lex.
\end{theorem}

We also prove that all finitely generated $3$-manifold groups lie in the family Lex.

\begin{corollary}\label{morelex}
Let $G$ be a finitely generated $3$-manifold group, then $G$ lies in the family Lex.
\end{corollary}

The proofs of Theorems \ref{polyfree} and \ref{lex} are both based on constructions of sequences of subgroups of $\pi_1(M)$, and the precise results are Theorems \ref{polyfreesubgroups} and \ref{lexsubgroups}. The proof of Theorem \ref{polyfreesubgroups} is easier, since we allow free group quotients in all steps; but Theorem \ref{lexsubgroups} is harder, since we only allow free groups in the last step.

The organization of this paper is summarized as follows. In Section \ref{preliminarygraph}, we review basic properties of graph $3$-manifolds. In Section \ref{polyfreeproof}, we use these properties to prove Theorem \ref{polyfree}, and Corollary \ref{morepolyfree} is also proved there. In Section \ref{lexbasics}, we review the definition and basic properties of Lex groups. We give some technical results on graph $3$-manifolds in Section \ref{moregraph}, mostly following \cite{WY}. In Section \ref{technicalsection}, we use the results in Section \ref{moregraph} to construct a finite cover of a closed, non-virtually fibered graph $3$-manifold satisfying some technical conditions (Proposition \ref{technicalresult}). Theorem \ref{lex} and Corollary \ref{morelex} will be proved in Sections \ref{lexproof} and \ref{morelexsection}, respectively.

\bigskip

{\bf Role of AI.} The proof of the second half of Lemma \ref{ai} was generated by Google Gemini, and checked by the author. The author also used Google Gemini to search for references and to translate \cite{Bra, Bou}.

{\bf Acknowledgment.} The author thanks Xiaolei Wu for asking him whether graph $3$-manifold groups are virtually poly-free. The author thanks Lvzhou Chen for asking him whether graph $3$-manifold groups lie in the family Lex, for teaching him about Lex groups, and for formulating a definition similar to Definition \ref{starn}, which in turn stems from Lvzhou's discussions with Thorben Kastenholz.

\section{Preliminary on graph $3$-manifolds}\label{preliminarygraph}

In this section, we review some basic properties of graph $3$-manifolds, which will be used in the proofs of Theorems \ref{polyfree} and \ref{lex}. We will review more technical results on graph $3$-manifolds in Section \ref{moregraph}.

\subsection{Slopes on the torus}\label{slopes}
We will use various notions of slopes on the torus $T^2$, and we set up some notations here. We fix an ordered basis $\{\alpha,\beta\}$ of the homology group $H_1(T^2;\mathbb{Z})\cong \mathbb{Z}^2$.



An {\it oriented multi-slope} $c$ on $T^2$ is an isotopy class of disjoint unions of oriented essential simple closed curves on $T^2$. Moreover, since all components of $c$ are parallel to each other, we also require that all components have consistent orientations. An oriented multi-slope corresponds to a unique non-zero element $m\alpha+n\beta\in H_1(T^2;\mathbb{Z})$, and we give it a coordinate $(m,n)\in \mathbb{Z}^2\setminus \{(0,0)\}$. If $c$ consists of a single curve, then $m$ and $n$ are coprime to each other, and we call $c$ an {\it oriented slope}, or simply a {\it slope}.

An {\it unoriented multi-slope} $c$ is obtained from an oriented multi-slope by forgetting the orientation. It corresponds to a non-zero element in $H_1(T^2;\mathbb{Z})$, up to multiplying $\pm1$, and is represented by $\pm(m,n)$ with $(m,n)\in \mathbb{Z}^2\setminus \{(0,0)\}$. We may also use $(m,n)$ to represent the unoriented multi-slope $c$ when no confusion is caused. Again, if $c$ consists of a single curve, we call it an {\it unoriented slope}.

We will use a pair of rational numbers $(r,s)\in \mathbb{Q}^2\setminus \{(0,0)\}$ to represent an {\it oriented rational slope}, which is a multiple of an oriented multi-slope by a positive rational number. 
We also have the notion of {\it unoriented rational slope}, which can be represented by $(r,s)\in \mathbb{Q}^2\setminus \{(0,0)\}$, when no confusion is caused.

The projective class of an oriented rational slope $(r,s)\in \mathbb{Q}^2\setminus \{(0,0)\}$ is uniquely represented by $\frac{r}{s}\in \mathbb{Q}\cup \{\infty\}\subset \mathbb{R}\cup \{\infty\}$. For any $2\times 2$ matrix $A=\begin{pmatrix}
a & b\\
c & d
\end{pmatrix}$ with rational entries, the action of $A$ on column vectors (orientated rational slopes) $(r,s)^t\in \mathbb{Q}^2\setminus \{(0,0)\}$ via the matrix-vector multiplication coincides the action of $A$ on rational numbers (projective slopes) $\frac{r}{s}\in \mathbb{Q}\cup \{\infty\}$ via the fractional linear transformation.

\subsection{Graph $3$-manifolds}\label{basicgraph}
In this paper, a graph $3$-manifold is a compact, connected, orientable, irreducible $3$-manifold with empty or tori boundary, such that it has a nontrivial JSJ decomposition $\mathcal{T}\subset M$ and each component of $M\setminus\setminus \mathcal{T}$ is a Seifert $3$-manifold.
Here we use $M\setminus\setminus \mathcal{T}$ to denote the compact $3$-manifold obtained by taking the path-metric completion of the open manifold $M\setminus \mathcal{T}$.
Each component of $M\setminus\setminus \mathcal{T}$ is called a {\it vertex manifold} of $M$.

The graph $3$-manifold $M$ has a dual graph $\Gamma_M$, where each vertex $v$ of $\Gamma_M$ corresponds to a vertex manifold $M_v\subset M$. Each edge $|e|$ of $\Gamma_M$ connecting vertices $v$ and $w$ corresponds to a JSJ torus $T_{|e|}\subset \mathcal{T}$ adjacent to $M_v$ and $M_w$. 

We will not review the definition of Seifert $3$-manifolds here. Since we only care about $\pi_1(M)$ up to finite-index subgroups, the readers can take the $\tilde{M}$ in the following lemma as a pseudo-definition of graph $3$-manifolds in this paper.

\begin{lemma}\label{simplegraph}
Let $M$ be a graph $3$-manifold, then $M$ has a finite cover $\tilde{M}$ such that the following hold.
\begin{enumerate}
\item The dual graph of $\tilde{M}$ has no self-loop (an edge with identical endpoints). 
\item Each vertex manifold of $\tilde{M}$ is homeomorphic to $F\times S^1$, where $F$ is a compact orientable surface with genus at least $1$ and has at least three boundary components.
\end{enumerate}
\end{lemma}
This result is well known, so we only give a sketchy proof here.

\begin{proof}
By Step I of the proof of \cite[Proposition 5.1]{SW}, $M$ has a finite cover $M'$ such that each vertex manifold of $M'$ is homeomorphic to $\Sigma\times S^1$ for some compact, orientable surface $\Sigma$ with genus at least $2$ and has non-empty boundary.

Since $\Sigma$ has genus at least $2$, it has a degree-$3$ cyclic cover $F$ with at least $3$ boundary components, such that the covering map restricts to each boundary component of $F$ as a homeomorphism to its image. Then $F\times S^1$ is a degree-$3$ cyclic cover of $\Sigma\times S^1$, and we can paste them to get a degree-$3$ cover $M''\to M'$. Then each vertex manifold of $M''$ has at least three boundary components and satisfies item (2).

If the dual graph of $M''$ has self-loops, we can take a double cover $\tilde{M}\to M''$ to get rid of them. Indeed, we can cut $M''$ along all JSJ tori corresponding to self-loops in $\Gamma_{M''}$, and paste two copies of the resulting manifold to get $\tilde{M}$.
\end{proof}

In the following, we assume that all graph $3$-manifolds satisfy the conclusion of Lemma \ref{simplegraph}, unless otherwise stated. We still denote such a graph $3$-manifold by $M$.

We fix an orientation on $M$, then each vertex manifold $M_v\subset M$ has an induced orientation. We also fix an orientation-preserving homemorphism between $M_v$ and $F_v\times S^1$, where $F_v$ is an oriented surface and $F_v\times S^1$ is equipped with the product orientation. Note that such a product structure on $M_v$ is not unique. 

\begin{definition}
For a graph $3$-manifold $M$ (satisfying the conclusion of Lemma \ref{simplegraph}), if each vertex manifold is equipped with a product structure with orientation data as above, then $M$ is called a {\it framed graph $3$-manifold}. 
\end{definition}
We assume that $M$ is a framed graph $3$-manifold in the following, unless otherwise stated.

Let $F$ be a compact, orientable surface with non-empty boundary. Let $q: N\to F\times S^1$ be a finite cover, then $N$ has an induced $S^1$-bundle structure with orientable base surface $\Sigma$, so $N$ is homeomorphic to $\Sigma \times S^1$. However, the product structure of $F\times S^1$ may not lift to a product structure of $N$, since a component of $q^{-1}(F)\subset N=\Sigma\times S^1$ may not project to $\Sigma$ by homeomorphism. For graph $3$-manifolds, if $M$ is a framed graph $3$-manifold and $p:\tilde{M}\to M$ is a finite cover, $\tilde{M}$ may not have an induced frame structure either. So we have the following definition.

\begin{definition}\label{induceframe}

Let $N$ be a Seifert $3$-manifold equipped with a product structure $N=\Sigma\times S^1$, and let $q: N = \Sigma \times S^1\to F\times S^1$ be a covering map. We say that $q$ is a {\it frame preserving finite cover} if it preserves the product structures, i.e. $q$ is the product of two covering maps $\Sigma\to F$ and $S^1\to S^1$.

 Let $M,\tilde{M}$ be two framed graph $3$-manifolds and let $p:\tilde{M}\to M$ be a finite cover. If $p$ restricts to each vertex manifold of $\tilde{M}$ as a frame preserving finite cover to the corresponding vertex manifold of $M$, we say that $p$ is a {\it frame preserving finite cover} between graph $3$-manifolds. In this case, the frame structure on $\tilde{M}$ is uniquely determined by the frame structure on $M$ and the covering map $p$.
\end{definition}

Let $n$ be the number of vertices of $\Gamma_M$. We take a bijection between the set of vertices of $\Gamma_M$ and $\{1,\cdots,n\}$ to label these vertices. If a vertex is labeled by $i$, we denote it by $v_i$, and denote the corresponding vertex manifold of $M$ by $M_i$. We use an arbitrary labeling for now and will not adopt any special labeling until Lemma \ref{tree}.

For an edge $|e|$ in $\Gamma_M$ between $v_i$ and $v_j$ ($i\ne j$ by Lemma \ref{simplegraph} (1)), we give it an arbitrary orientation and denote the oriented edge by $e$, then the orientation reversal of $e$ is denoted by $\bar{e}$. If the initial point of $e$ is $v_i$ and the terminal point is $v_j$, we write $e:v_i\to v_j$. If the initial vertex of $e$ is $v_i$ and we do not care about the terminal point, we write $e:v_i\to $.

An oriented edge $e:v_i\to v_j$ ($i\ne j$) in $\Gamma_M$ corresponds to two boundary components of $M\setminus\setminus \mathcal{T}$. 
We use $T_e$ and $T_{\bar{e}}$ to denote these two oriented boundary components of $M\setminus\setminus \mathcal{T}$, with $T_e\subset \partial M_i$ and $T_{\bar{e}}\subset M_j$. When we paste $M\setminus\setminus \mathcal{T}$ back to obtain $M$, we use an orientation-reversing homeomophism $f_e:T_e\to T_{\bar{e}}$, and we have $f_{\bar{e}}=(f_e)^{-1}$.

On $T_e\subset \partial M_i=F_i\times S^1$, the oriented boundary of $F_i$ and the oriented $S^1$ give two oriented slopes $\alpha_e,\beta_e$, and they form an ordered basis of $H_1(T_e;\mathbb{Z})\cong \mathbb{Z}^2$. Similarly, we equip $T_{\bar{e}}\subset \partial M_j=F_j\times S^1$ with two oriented slopes $\alpha_{\bar{e}},\beta_{\bar{e}}$ that form an ordered basis of $H_1(T_{\bar{e}};\mathbb{Z})\cong \mathbb{Z}^2$. The induced isomorphism $(f_e)_{\#}:H_1(T_e;\mathbb{Z})\to H_1(T_{\bar{e}};\mathbb{Z})$ is represented by a $2\times 2$ matrix with determinant $-1$:
\begin{align}\label{2.1}
(f_e)_{\#}
\begin{pmatrix}
\alpha_e\\
\beta_e
\end{pmatrix}=
\begin{pmatrix}
p_e & q_e \\
r_e & s_e
\end{pmatrix}
\begin{pmatrix}
\alpha_{\bar{e}}\\
\beta_{\bar{e}}
\end{pmatrix}.
\end{align}
Note that $r_e\ne 0$ since $f_e$ does not match the $S^1$-fibers of $M_i$ and $M_j$ (i.e. does not send $\beta_e$ to $\pm\beta_{\bar{e}}$). Since $(f_{\bar{e}})_{\#}=(f_{e})_{\#}^{-1}$, a direct computation gives $r_e=r_{\bar{e}}$.

Any oriented rational slope $c$ on $T_e$ has homology class $u\alpha_e+t\beta_e$, with coordinate $(u,t)\in \mathbb{Q}^2\setminus \{(0,0)\}$.
Then $f_e(c)$ is an oriented rational slope on $T_{\bar{e}}$ with coordinate $(u',t')$ such that
\begin{align}\label{2.2}
\begin{pmatrix}
p_e & r_e \\
q_e & s_e
\end{pmatrix}
\begin{pmatrix}
u\\
t
\end{pmatrix}=
\begin{pmatrix}
u'\\
t'
\end{pmatrix}.
\end{align}
Here $\begin{pmatrix}
p_e & r_e \\
q_e & s_e
\end{pmatrix}$ is the transpose of the matrix in \eqref{2.1}, and we abuse notation to denote it by $f_e\in M_{2\times 2}(\mathbb{Q})$.

A {\it horizontal subsurface} $\Sigma$ in a product Seifert $3$-manifold $F\times S^1$ is an orientable, properly embedded subsurface that is transverse to the $S^1$-fibers everywhere, which is also a fiber surface of $F\times S^1$. The following lemma can be found in \cite[Lemma 1.1]{WY}.
 
\begin{lemma}\label{horizontalsurface}
Let $F\times S^1$ be a product Seifert $3$-manifold such that $F$ has genus at least $1$, and let $T_1,\cdots, T_n$ be all boundary components of $F\times S^1$. Let $c_i$ be an oriented multi-slope on $T_i$ consisting of $k_i$ copies of the oriented rational slope  $(u_i,t_i)$, where $k_i$ is a positive integer. There exists an oriented horizontal subsurface $\Sigma \subset F\times S^1$  such that $\Sigma\cap T_i=c_i$ as oriented multi-slopes if and only if
\begin{enumerate} 
\item all the $u_i$ have the same sign and $k_iu_i=\lambda\ne 0$ is independent of $i$,
\item $k_iu_i,k_it_i$ are integers for all $i=1,\cdots,n$,
\item $\sum_{i=1}^n\frac{t_i}{u_i}=0$.
\end{enumerate}
Here we equip $\Sigma\cap T_i$ with the boundary orientation induced from the orientation of $\Sigma$.
\end{lemma}

The statement of Lemma \ref{horizontalsurface} only requires that $(u_i,t_i)$ is a rational slope, while \cite[Lemma 1.1]{WY} requires that $(u_i,t_i)$ is an (integer) slope, so we need item (2) in Lemma \ref{horizontalsurface}.

\begin{remark}\label{connected} 
Actually, we can make $\Sigma$ connected. If $\Sigma$ has more than one component, since it is horizontal in $F\times S^1$ and intersects with each boundary component along parallel curves with consistent orientations, $\Sigma$ consists of parallel copies of a connected, oriented, horizontal subsurface. Since $F$ has genus at least $1$, there exists a non-separating simple closed curve $\gamma \subset F$. Then $T=\gamma \times S^1$ is a torus that is transverse to $\Sigma$, and $T\cap \Sigma$ consists of $m$ copies of parallel curves for some positive integer $m$. We cut $\Sigma$ along $T$ and re-paste the $m$ pairs of curves by a $1$-shift, to get a connected oriented horizontal subsurface in $F\times S^1$ with the same boundary slopes as $\Sigma$.
\end{remark}

Now we review non-positively curved metrics and the virtual fiberness of graph $3$-manifolds. More details on this topic can be found in \cite{BS}.

For graph $3$-manifolds with boundary, they always admit non-positively curved metrics (\cite{Leeb}) and are always virtually fibered (\cite{WY}). So we mainly focus on closed graph $3$-manifolds in this paper. In the following, $M$ will denote a closed framed graph $3$-manifold satisfying the conclusion of Lemma \ref{simplegraph}.


Let $M_i$ be the vertex manifold of $M$ corresponding to a vertex $v_i$ of $\Gamma_M$. Let $e:v_i\to v_j$ be an oriented edge in $\Gamma_M$, and let $c_e$ be the oriented slope on $T_{e}$ given by the $S^1$-fiber of $M_j$ (which is the $\beta_{\bar{e}}$ in equation \eqref{2.1}). By equation \eqref{2.1}, the coordinate of $c_e$ with respect to $M_i$ is $(r_e,-p_e)$. The {\it charge} of $M_i\subset M$ is defined by 
\begin{align}\label{2.3}
z_i=\sum_{e:v_i\to } \frac{p_e}{r_e},
\end{align}
where the sum runs over all oriented edges in $\Gamma_M$ with initial vertex $v_i$.
Note that $z_i$ is independent of the orientation of $c_e$, and it equals the Euler number of the closed Seifert manifold obtained by Dehn filling $M_i$ along all $c_e$. This definition of charge may differ from that in some other papers by a sign, but that does not matter, since we only care whether the charge is zero.

A closed framed graph $3$-manifold is called {\it chargeless} if $z_i=0$ for all vertices $v_i$ in $\Gamma_M$. Then we have the following relation among chargelessness, virtual fiberness, and non-positively curved metrics.
\begin{theorem}\label{chargeless}
Let $M$ be a closed framed graph $3$-manifold.
\begin{enumerate}
\item If $M$ is chargeless, then it admits a non-positively curved metric.
\item If $M$ admits a non-positively curved metric, then $M$ virtually fibers over the circle.
\end{enumerate}
\end{theorem}

Theorem \ref{chargeless} (1) follows from \cite[Theorem 4.9]{BS}. The chargeless condition implies the function $s:U\to \{0,\pm1\}$ in \cite[Theorem 4.9]{BS} is zero. Theorem \ref{chargeless} (2) is exactly \cite[Corollary 2.4]{BS}.


\section{Graph $3$-manifold groups are virtually poly-free}\label{polyfreeproof}

In this section, we will prove Theorem \ref{polyfree} and Corollary \ref{morepolyfree}.
The proof of Theorem \ref{polyfree} is given by the following result.

\begin{theorem}\label{polyfreesubgroups}
Let $M$ be a compact graph $3$-manifold, then $G=\pi_1(M)$ contains a sequence of subgroups 
$$G>G_1\vartriangleright G_2 \vartriangleright G_3,$$ such that the following hold.
\begin{enumerate}
\item $G_1$ is a finite-index subgroup of $G$.
\item $G_1/G_2$ is isomorphic to a finitely generated free group.
\item $G_2/G_3\cong \mathbb{Z}$.
\item $G_3$ is isomorphic to a free group (possibly infinitely generated).
\end{enumerate}
\end{theorem}

\begin{proof}
{\bf Step I. } Construction of $G_1<G$ with finite-index. 

We take a finite cover $M_1\to M$ as given in Lemma \ref{simplegraph}, and let $G_1=\pi_1(M_1)$. We equip $M_1$ with a frame structure. By Lemma \ref{simplegraph} (2), each vertex of the dual graph $\Gamma_{M_1}$ has valence at least $3$. So $\pi_1(\Gamma_{M_1})$ is a finitely generated non-abelian free group. 

{\bf Step II.} Construction of $G_2\vartriangleleft G_1$ such that $G_1/G_2$ is free. 

Let $q_1:M_1\to \Gamma_{M_1}$ be the natrual projection, let $G_2$ be the kernel of $(q_1)_*:\pi_1(M_1)\to \pi_1(\Gamma_{M_1})$, and let $p_2:M_2\to M_1$ be the infinite-sheet cover of $M_1$ corresponding to $G_2<G_1=\pi_1(M_1)$. Then $G_1/G_2\cong \pi_1(\Gamma_{M_1})$ is free.

Let $\mathcal{T}_1\subset M_1$ be the torus decomposition of $M_1$, then $\mathcal{T}_2=p_2^{-1}(\mathcal{T}_1)\subset M_2$ is an infinite union of disjoint tori in $M_2$. Each vertex manifold $M_v$ of $M_2\setminus\setminus \mathcal{T}_2$ is a compact product manifold $F\times S^1$ that maps to $M_1$ by a homeomorphism to its image. So $M_2$ has an induced frame structure. The dual graph $\Gamma_{M_2}$ of $M_2$ with respect to $\mathcal{T}_2$ is the universal cover of $\Gamma_{M_1}$, which is an infinite tree.

{\bf Step III.} Construction of a free subgroup $G_3\vartriangleleft G_2$ such that $G_2/G_3\cong \mathbb{Z}$.

We will construct a connected, non-compact, oriented, properly embedded subsurface $\Sigma\subset M_2$, such that it intersects with each vertex manifold of $M_2$ at a connected horizontal subsurface. 

We start with an arbitrary vertex $v_1$ of $\Gamma_{M_2}$ and let $\{v_1\}=K_1\subset \cdots \subset K_i\subset \cdots \subset \Gamma_{M_2}$ be a sequence of (connected) finite subtrees of $\Gamma_{M_2}$, such that $K_i\setminus K_{i-1}$ has only one vertex and $\cup_{i=1}^{\infty}K_i=\Gamma_{M_2}$. For each $i$, let $M_{2,i}$ be the compact submanifold of $M_2$ corresponding to $K_i$, then we want to inductively construct a connected, oriented, properly embedded subsurface $\Sigma_i\subset M_{2,i}$, such that the following hold.
\begin{enumerate}
\item[(i)] $\Sigma_i$ intersects with each vertex manifold of $M_{2,i}$ at a connected horizontal subsurface.
\item[(ii)] For each boundary component $T\subset \partial M_{2,i}$, $\Sigma_i$ intersects with $T$ at an oriented multi-slope that is not a multiple of $S^1$-fibers of either vertex manifolds in $M_2$ adjacent to $T$.
\item[(iii)] $\Sigma_i\cap M_{2,i-1}=\Sigma_{i-1}$.
\end{enumerate} 

We start with $K_1=\{v_1\}$ and denote the corresponding vertex manifold by $M_{v_1}=M_{2,1}=F_{v_1}\times S^1$. Let $T_1, \cdots,T_n$ be all boundary components of $M_{v_1}$, and let $f_i$ be the oriented slope on $T_i$ given by the $S^1$-fiber of the vertex manifold adjacent to $M_{2,1}$ along $T_i$. By Lemma \ref{horizontalsurface} and Remark \ref{connected}, oriented slopes $(1,t_1),(1,t_2),\cdots,(1,t_n)$ on $T_1, T_2,\cdots,T_n$ bound a connected, oriented, horizontal subsurface $\Sigma_1$ in $M_{v_1}=M_{2,1}$ if and only if $t_1+\cdots+ t_n=0$. Since $n\geq 3$ (Lemma \ref{simplegraph} (2)), we can take $(t_1,\cdots,t_n)=l(-n+1, 1,\cdots, 1)$ for some $l\in \mathbb{Z}$ such that $(1,t_i)\ne \pm f_i$ holds for all $i$, to construct the desired subsurface $\Sigma_1\subset M_{2,1}$.

Suppose that we have constructed $\Sigma_i\subset M_{2,i}$, then we want to construct $\Sigma_{i+1}\subset M_{2,i+1}$. Let $v$ be the unique vertex in $K_{i+1}\setminus K_i$, and let the corresponding vertex manifold in $M_2$ be $M_v=F_v\times S^1$. Since $\Gamma_{M_2}$ is a tree, $M_{2,i}\cap M_v=T_0'$ is a single torus, and $M_{2,i+1}=M_{2,i}\cup M_v$. Let the boundary components of $M_v$ be $T_0', T_1',\cdots,T_n'$, then Lemma \ref{simplegraph} (2) implies $n\geq 2$.

By the induction hypothesis (ii), $\Sigma_i\cap T_0'$ is not a multiple of the $S^1$-fiber of $M_v$, so its
coordinate $(u,t_0)$ with respect to $M_v=F_v\times S^1$ satisfies $u\ne 0$. Let $f_i'$ be the oriented slope on $T_i'$ given by the $S^1$-fiber of the vertex manifold adjacent to $M_v$ along $T_i'$. By Lemma \ref{horizontalsurface} and Remark \ref{connected}, together with $(u,t_0)$ on $T_0'$, oriented multi-slopes $(u,t_1),\cdots,(u,t_n)$ on $T_1',\cdots,T_n'$ bound a connected, oriented, horizontal subsurface $\Sigma_v$ in $M_v$ if and only if $t_0+t_1+\cdots+ t_n=0$. Since $n\geq 2$, we can take $(t_1,\cdots,t_n)=l(-n+1,1,\cdots,1)+(0,0,\cdots,-t_0)$ for some $l\in \mathbb{Z}$, such that $(u,t_i)$ is not a multiple of $f_i'$ for all $i$. Then $\Sigma_{i+1}=\Sigma_i\cup (-\Sigma_v)$ is the desired oriented subsurface in $M_{2,i+1}$. Here, we use $-\Sigma_v$ because we need an orientation-reversing homeomorphism to paste two oriented manifolds along their boundaries. Since $\Sigma_v$ and $\Sigma_i$ are connected, $\Sigma_{i+1}$ is also connected.

This finishes the inductive construction of the subsurface $\Sigma_{i+1}\subset M_{2,i+1}$. Then $\Sigma=\cup_{i=1}^{\infty}\Sigma_i$ is the desired subsurface of $M_2$.

By taking the oriented intersection number with the connected, non-compact, oriented, properly embedded subsurface $\Sigma\subset M_2$, we get a surjective homomorphism $q_2:\pi_1(M_2)\to H_1(M_2;\mathbb{Z})\to \mathbb{Z}$. Alternatively, for each vertex manifold $M_v\subset M_2$, $\Sigma\cap M_v$ is a connected, oriented fiber surface of $M_v$ and it gives a surjective homomorphism $H_1(M_v;\mathbb{Z})\to \mathbb{Z}$. Since all fiber surfaces agree on tori in $\mathcal{T}_2$ and $\Gamma_{M_2}$ is a tree, a Mayer–Vietoris sequence argument gives the same homomorphism $q_2:H_1(M_2;\mathbb{Z})\to \mathbb{Z}$. Although the second viewpoint looks more complicated, it will be useful in our proof of Theorem \ref{lex}.

Let $M_3$ be the covering space of $M_2$ corresponding to $\text{ker}(q_2)$. Then $G_3=\pi_1(M_3)\vartriangleleft G_2=\pi_1(M_2)$ such that $G_2/G_3\cong \mathbb{Z}$, and we check that $G_3$ is isomorphic to an (infinitely generated) free group. For any vertex manifold $M_v\subset M_2$, since $\Sigma_v = \Sigma \cap M_v$ is connected, $q_2|_{\pi_1(M_v)}:\pi_1(M_v)\to \mathbb{Z}$ is surjective.  Then $q_2^{-1}(M_v)\subset M_3$ is connected and is homeomorphic to $\Sigma_v \times \mathbb{R}$. These product manifolds $\Sigma_v \times \mathbb{R}$ are pasted in $M_3$ along a union of cylinders $(\Sigma \cap \mathcal{T}_2)\times \mathbb{R}$. So $M_3$ is homotopy equivalent to the space obtained by pasting compact surfaces $\Sigma_v$ along their boundaries, which is just the non-compact surface $\Sigma$. So $G_3=\pi_1(M_3)\cong \pi_1(\Sigma)$ is isomorphic to an infinitely generated free group.

The proof of Theorem \ref{polyfreesubgroups} is done.
\end{proof}

Theorem \ref{polyfreesubgroups}  implies that the $G_1$ in its statement is poly-free, so $G=\pi_1(M)$ is virtually poly-free, thus Theorem \ref{polyfree} holds true.

To prove Corollary \ref{morepolyfree}, we first prove the following elementary lemma. This lemma is well-known, and we include a proof here for completeness.

\begin{lemma}\label{polyfreeproperty}
\begin{enumerate}
\item Let $G$ be a virtually poly-free group and let $H<G$ be a subgroup, then $H$ is virtually poly-free.
\item Let $1\to A\to G\xrightarrow{q} B\to 1$ be an exact sequence. If both $A$ and $B$ are poly-free, then $G$ is poly-free.
\item Let $A,B$ be two virtually poly-free groups, then the free product $A*B$ is virtually poly-free. 
\end{enumerate}
\end{lemma}

\begin{proof}
We first prove item (1). Let $G_0<G$ be a finite-index subgroup that is poly-free, and let $G_0\vartriangleright G_1 \vartriangleright \cdots \vartriangleright G_n=\{1\}$ be a sequence of subgroups such that $G_i/G_{i+1}$ is isomorphic to a free group. Then $H_0=H\cap G_0$ is a finite-index subgroup of $H$, and we have a sequence of subgroups
$$H_0=H\cap G_0> H\cap G_1 >\cdots > H\cap G_n=\{1\}.$$ 
It is clear that $H\cap G_{i+1}$ is a normal subgroup of $H\cap G_i$.
Since $H\cap G_i/H\cap G_{i+1}\cong (H\cap G_i)\cdot G_{i+1}/G_{i+1}<G_i/G_{i+1}$ and a subgroup of a free group is still free (possibly infinitely generated), $H$ is virtually poly-free.

Then we prove item (2). Let $A=A_0\vartriangleright A_1\vartriangleright\cdots \vartriangleright A_n=\{1\}$ and  $B=B_0\vartriangleright B_1\vartriangleright\cdots \vartriangleright B_m=\{1\}$ be sequences of subgroups that certify the poly-freeness of $A$ and $B$. We consider $A$ as a subgroup of $G$, then the following sequence of subgroups certifies the poly-freeness of $G$
$$G=q^{-1}(B_0)\vartriangleright q^{-1}(B_1) \vartriangleright \cdots \vartriangleright q^{-1}(B_m)=A=A_0 \vartriangleright A_1 \vartriangleright \cdots \vartriangleright A_n=\{1\}.$$

Now we prove item (3). We first prove that if $A$ and $B$ are both poly-free, then $A*B$ is poly-free. Let $A*B\to A\times B$ be the natural projection to the direct product, and let $K$ be the kernel. Then we have an exact sequence
$$1\to K\to A*B\to A\times B\to 1.$$ 
The Kurosh subgroup theorem implies that $K$ is a free group (possibly infinitely generated). By item (2), the poly-freeness of $A$ and $B$ implies that $A\times B$ is poly-free. Then the above exact sequence and the poly-freeness of $K$ imply that $A*B$ is poly-free, by applying item (2) again.

Now we suppose that $A$ and $B$ are both virtually poly-free. Let $A_0<A$ and $B_0<B$ be finite-index subgroups that are poly-free, and by the proof of item (1), we can assume they are both normal subgroups. Let $A*B\to A\times B\to A/A_0\times B/B_0$ be the natural projection, and let $H$ be the kernel, which is a finite-index subgroup of $A*B$. The Kurosh subgroup theorem implies that $H$ is a free product of finitely many copies of $A_0$ and $B_0$ and a finitely generated free group. The previous paragraph implies that $H$ is poly-free, so $A*B$ is virtually poly-free.
\end{proof}

Now we prove Corollary \ref{morepolyfree}, which implies that all finitely generated $3$-manifold groups are virtually poly-free.

\begin{proof}
Let $M$ be a (connected) $3$-manifold with a finitely generated fundamental group, and we want to prove $G=\pi_1(M)$ is virtually poly-free.

{\bf Step I.} We first suppose that $M$ is a compact, orientable, irreducible $3$-manifold with empty or tori boundary. 

If $M$ has a finite fundamental group, then $G=\pi_1(M)$ is virtually trivial, thus virtually poly-free.
If $M$ is not a graph $3$-manifold, Theorem \ref{virtualfiber} implies that $M$ has a finite cover $\tilde{M}$ that is a fiber bundle. Since surface groups are poly-free, as discussed in the introduction, Lemma \ref{polyfreeproperty} (2) implies that $\pi_1(\tilde{M})$ is poly-free, thus $G=\pi_1(M)$ is virtually poly-free. If $M$ is a graph $3$-manifold, then Theorem \ref{polyfree} implies that $G=\pi_1(M)$ is virtually poly-free.

{\bf Step II.} Now we suppose that $M$ is compact, orientable, irreducible,a and $\partial$-irreducible. 

Note that the irreducibility implies that $M$ has no $S^2$ boundary component, unless $M=D^3$ with a trivial fundamental group.

For each boundary component $\Sigma$ of $M$ with genus $g\geq 2$, we take a compact hyperbolic $3$-manifold $N_g$ with connected, totally geodesic boundary homeomorphic to $\Sigma$. Then we paste $N_g$ to $M$ by a homeomorphism between $\partial N_g$ and $\Sigma$, to get a $3$-manifold $N$ as in Step I (see \cite[Section 6.3]{Sun}). The inclusion of manifold $M\to N$ induces an injective homomorphism $\pi_1(M)\to \pi_1(N)$. Step I implies that $\pi_1(N)$ is virtually poly-free, then Lemma \ref{polyfreeproperty} (1) implies that $G=\pi_1(M)$ is virtually poly-free.

{\bf Step III.} Then we suppose that $M$ is compact, orientable, and irreducible. 

We can inductively compress $\partial M$ along properly embedded discs to obtain a disjoint union of $3$-manifolds as in Step II. So $\pi_1(M)$ is a free product of virtually poly-free groups and a free group. Then Lemma \ref{polyfreeproperty} (3) implies that $\pi_1(M)$ is virtually poly-free.

{\bf Step IV.} We suppose that $M$ is compact and orientable.

The prime decomposition implies that $M$ is a connected sum of manifolds as in Step III and copies of $S^2\times S^1$. So $\pi_1(M)$ is a free product of virtually poly-free groups and a free group, and it is virtually poly-free by Lemma \ref{polyfreeproperty} (3) again.

{\bf Step V.} We suppose that $M$ is an arbitrary $3$-manifold with finitely generated fundamental group. 

If $M$ is non-orientable, we take the orientable double cover and still denote the manifold by $M$, since we only care about a virtual property. If $M$ is not compact, we take the Scott core $N\subset M$ (\cite{Sco}), which is a compact codimension-$0$ submanifold such that the inclusion-induced homomorphism $\pi_1(N)\to \pi_1(M)$ is an isomorphism. By Step IV, $\pi_1(N)$ is virtually poly-free, so $\pi_1(M)$ is virtually poly-free.

We finish the proof of Corollary \ref{morepolyfree}.
\end{proof}

\section{Preliminary on Lex groups}\label{lexbasics}

In this section, we review some basic properties of Lex groups. The material can be found in \cite{Bou} and \cite{FLM}.

The family Lex of groups is defined via bounded cohomology in \cite{Bou}. This notion is called class $\Lambda$ in \cite{Bou}, and is called family Lex in \cite{FLM}. 
\begin{definition}\label{lexdef}
A group $\Gamma$ lies in the {\it family Lex} if it satisfies the following left-exactness property. For any group $\Lambda$ and any surjective homomorphism $\psi:\Lambda\to \Gamma$, the induced homomorphism $H_b^n(\psi): H_b^n(\Gamma)\to H_b^n(\Lambda)$ in bounded cohomology with trivial real coefficients is injective for all $n$. In this case, we also say that $\Gamma$ is a {\it Lex group}.
\end{definition}

Here $H_b^n(\Gamma)$ denotes the $n$-th bounded cohomology group of $\Gamma$. We will not really work with bounded cohomology of groups in this paper, so we will not review its definition. Instead, we will use the following properties of Lex groups established in \cite{Bou}.

\begin{proposition}\label{lexproperty}
\begin{enumerate}
\item Free groups (finitely or infinitely generated) and surface groups lie in the family Lex (\cite[Remark 7 and Proposition 3.8]{Bou}).
\item If $\Gamma'<\Gamma$ is a finite-index subgroup and $\Gamma'$ lies in the family Lex, then $\Gamma$ lies in the family Lex (\cite[Theorem 3.6]{Bou}).
\item For an exact sequence $$1\to \Gamma \to \Pi \to A\to 1,$$ if $\Gamma$ lies in the family Lex and $A$ is amenable, then $\Pi$ lies in the family Lex  (\cite[Theorem 3.2]{Bou}).
\end{enumerate}
\end{proposition}

We will only use properties of Lex groups in Proposition \ref{lexproperty} to prove Theorem \ref{lex} and Corollary \ref{morelex}. It is unknown whether free products of Lex groups lie in the family Lex (see page 257 of \cite{Bou}).

Proposition \ref{lexproperty} implies that all $3$-manifolds that are virtually fiber bundles have Lex fundamental groups. For completeness, this well-known result will be proved in Section \ref{morelexsection}.





\section{More on graph $3$-manifolds}\label{moregraph}

In this section, we present some technical results on closed framed graph $3$-manifolds that use linear equations to construct subsurfaces. The results in this section either belong to \cite{WY} or are generalizations of results in \cite{WY}. 

\subsection{Construct subsurfaces by solving linear equations}\label{equations}

We first fix some notations. Let $M$ be a closed framed graph $3$-manifold, let the vertices in $\Gamma_M$ be $v_1,\cdots, v_n$, and  let the corresponding vertex manifolds be $M_1,\cdots,M_n$. 

We will assign each $M_i$ a non-zero rational number $\lambda_i\in \mathbb{Q}\setminus \{0\}$. For any oriented edge $e:v_i\to v_j$, recall that $r_e\in \mathbb{Z}\setminus \{0\}$ is defined in equation \eqref{2.1}, which equals the intersection number of $S^1$-fibers of $M_i$ and $M_j$ on $T_{|e|}$ (up to a sign). Let $\epsilon_e\in \{\pm1\}$ be the sign of $r_e$, then $r_e=r_{\bar{e}}$ implies $\epsilon_e=\epsilon_{\bar{e}}$. We equip $T_e$ and $T_{\bar{e}}$ with oriented rational slopes
\begin{align}\label{5.1}
(|r_e|\lambda_i, \lambda_j-\epsilon_ep_e \lambda_i)\text{\ and\ }(r_e\lambda_j,\epsilon_e\lambda_i+s_e\lambda_j),
\end{align}
respectively. It is easy to check that
$$
f_e
\begin{pmatrix}
|r_e|\lambda_i\\ 
\lambda_j-\epsilon_ep_e \lambda_i
\end{pmatrix}=
\begin{pmatrix}
r_e\lambda_j\\
\epsilon_e\lambda_i+s_e\lambda_j
\end{pmatrix},
$$
where $f_e=\begin{pmatrix}
p_e & r_e\\ 
q_e & s_e
\end{pmatrix}$ is given in equation \eqref{2.2}. If we flip the orientation of $e$, this assignment of oriented rational slopes changes by multiplying $\epsilon_e\in \{\pm 1\}$. So it gives a well-defined unoriented rational slope on $T_{|e|}$.

The following result is a restatement of \cite[Lemma 1.4]{WY} in our context.
\begin{lemma}\label{horizontalbasic}
Let $M$ be a closed framed graph $3$-manifold. If there are non-zero rational numbers $\lambda_1,\cdots,\lambda_n$ such that 
\begin{align}\label{5.2}
\sum_{\{(e,j)|e:v_i\to v_j\}}\Big(\frac{\lambda_j}{ |r_e|\lambda_i}-\frac{p_e}{r_e}\Big)=0
\end{align}
holds for all $i$, then there exists a closed, connected, embedded subsurface $\Sigma\to M$ such that the following hold.
\begin{enumerate}
\item For any vertex manifold $M_i$ of $M$, $\Sigma\cap M_i$ is a connected, orientable, horizontal subsurface of $M_i$.
\item For any JSJ torus $T_{|e|}$ of $M$, $\Sigma\cap T_{|e|}$ is a non-zero multiple of the slope assignment \eqref{5.1}, as unoriented multi-slopes.
\end{enumerate}
In this case, $M$ is a virtually fibered graph $3$-manifold.
\end{lemma}

Since the assignment of oriented rational slopes in \eqref{5.1} is only well-defined up to a sign, $\Sigma$ may not be orientable. For an embedded subsurface of a graph $3$-manifold $M$ satisfying condition (1), we call it a {\it horizontal subsurface} of $M$.

The proof of this lemma will be used in the proof of Lemma \ref{horizontaladvanced}, so we prove it here.

\begin{proof}

We take a large integer $N$ such that $N\lambda_i$, $\frac{N}{r_e}$, $\frac{N\lambda_i}{r_e}$ are all integers for all vertice $v_i$ and edge $|e|$ in $\Gamma_M$. For any oriented edge $e:v_i\to v_j$, we equip $T_e\subset \partial M_i$ with the oriented rational slope $(|r_e|\lambda_i, \lambda_j-\epsilon_ep_e \lambda_i)$ in equation \eqref{5.1}. 

Now we work on a vertex manifold $M_i$ of $M$. For any oriented edge $e:v_i\to $ in $\Gamma_M$, we take $k_e=\frac{N}{|r_e|}$, then $k_e=k_{\bar{e}}$ holds. Our choice of integers and rational slopes for $e:v_i\to$ clearly satisfies Lemma \ref{horizontalsurface} (1) and (2). The equation in Lemma \ref{horizontalsurface} (3) is 
$$\sum_{\{(e,j)|e:v_i\to v_j\}}\frac{\lambda_j-\epsilon_ep_e \lambda_i}{|r_e|\lambda_i}=\sum_{\{(e,j)|e:v_i\to v_j\}}\Big(\frac{\lambda_j}{|r_e|\lambda_i}-\frac{p_e}{r_e}\Big),$$
which equals $0$ by equation \eqref{5.2}. Then Lemma \ref{horizontalsurface} implies that these oriented multi-slopes bound a connected, oriented, horizontal subsurface $\Sigma_i\subset M_i$.

Since our choice of oriented rational slope on $T_{|e|}$ is well-defined up to a sign and $k_e=k_{\bar{e}}$, for any $e:v_i\to v_j$, $\Sigma_i \cap T_{|e|}$ and $\Sigma_j\cap T_{|e|}$ are equal to each other as unoriented multi-slopes. So the subsurfaces $\Sigma_1,\cdots,\Sigma_n$ can be pasted together along JSJ tori of $M$ to a closed, embedded subsurface $\Sigma\to M$. Since each $\Sigma_i$ is connected, the surface $\Sigma$ is connected. So $\Sigma$ satisfies conditions (1) and (2).
Since we only paste along unoriented multi-slopes, $\Sigma$ might be non-orientable. 

Now we prove that $M$ is virtually fibered.

We first assume that $\Sigma$ is orientable. For any vertex manifold $M_i\subset M$, since $M_i\cap \Sigma$ is horizontal in $M_i$, $M_i\setminus\setminus (M_i\cap \Sigma)$ has an $I$-bundle structure. These $I$-bundle structures can be pasted to an $I$-bundle structure on $M\setminus\setminus \Sigma$. Since $\Sigma$ is connected and orientable, $M\setminus \setminus \Sigma$ has two boundary components. So $M\setminus\setminus \Sigma$ is homeomorphic to either $\Sigma\times I$ or two orientable twisted $I$-bundles over non-orientable surfaces. Then $M$ is fibered in the first case, and has a fibered double cover in the second case.

If $\Sigma$ is non-orientable, a neighborhood $\mathcal{N}(\Sigma)$ of $\Sigma$ in $M$ is an orientable twisted $I$-bundle over $\Sigma$. Then $\partial \mathcal{N}(\Sigma)$ is an orientable, connected, horizontal subsurface in $M$, and the proof reduces to the previous case.
\end{proof}

The equation \eqref{5.2} can be written in a matrix form. Let 
$$y_{ij}=\sum_{e:v_i\to v_j}\frac{1}{|r_e|}\text{\ and\ }z_i=\sum_{e:v_i\to }\frac{p_e}{r_e}.$$ 
We have $y_{ij}=y_{ji}\geq 0$, where the equality holds if and only if there exists $e:v_i\to v_j$. So Lemma \ref{simplegraph} (1) implies $y_{ii}=0$. Note that $z_i$ equals the charge of $M_i$ defined in equation \eqref{2.3}.  Then we define two rational symmetric $n\times n$ matrices by $Y=(y_{ij})_{n\times n}$ and $Z=\text{diag}(z_1,\cdots,z_n)$. Here $\text{diag}(z_1,\cdots,z_n)$ denotes the diagonal $n\times n$ matrix with diagonal entries $z_1,\cdots,z_n$.  Let $\Lambda=(\lambda_1,\cdots,\lambda_n)^t$ be a colomn vector, then equation \eqref{5.2} is equivalent to 
\begin{align}\label{5.3}
(Y-Z)\cdot \Lambda=\vec{0}.
\end{align} We say that $\Lambda$ is {\it totally non-zero} if all of its coordinates are not zero, and we say that $\Lambda$ is a {\it rational vector} if all of its coordinates are rational numbers.

For a closed framed graph $3$-manifold that is not virtually fibered, Lemma \ref{horizontalbasic} implies that equation \eqref{5.3} has no totally non-zero rational solution. The following lemma is a version of Lemma \ref{horizontalbasic} in this case.

\begin{lemma}\label{horizontaladvanced}
Let $M$ be a closed framed graph $3$-manifold. Let $e_*:v_{n-3}\to v_{n-1}$ and $e_{**}:v_{n-2}\to v_n$ be two oriented edges in $\Gamma_M$ connecting these vertices, such that $f_{e_*}=f_{e_{**}}=\begin{pmatrix}
p & r\\
q & s
\end{pmatrix}$.
We suppose that there is a totally non-zero rational vector $\Lambda=(\lambda_1,\cdots,\lambda_n)^t$ such that 
\begin{align}\label{5.4}
(Y-Z)\cdot \Lambda =
(0, \cdots, 0, \delta, -\delta, \delta', -\delta')^t
\end{align}
with $\lambda_{n-3}=-\lambda_{n-2}=\lambda$, $\lambda_{n-1}=-\lambda_n=\lambda'$, and $\delta,\delta'\in \mathbb{Q}$. Then there exists a compact surface $\Sigma$ with boundary and a map $f:\Sigma\to M$ such that the following hold.
\begin{enumerate}
\item The map $f:\Sigma\to M$ is an embedding in the interior of $\Sigma$.
\item $f^{-1}(T_{|e_*|}\cup T_{|e_{**}|})=\partial \Sigma$.
\item For any vertex manifold $M_i\subset M$, $\Sigma\cap M_i$ is a connected, orientable, horizontal subsurface of $M_i$.
\item For any JSJ torus $T_{|e|}\subset M$ with $|e|\ne |e_*|$ or $|e_{**}|$, $\Sigma\cap T_{|e|}$ is a multiple of the slope assignment \eqref{5.1}, as unoriented multi-slopes.
\item Let $N=M\setminus\setminus (T_{|e_*|}\cup T_{|e_{**}|})$, items (1) and (2) give a proper embedding $\Sigma\to N$. Then the intersection of $\Sigma$ with $T_{e_*},T_{e_{**}},T_{\bar{e}_*},T_{\bar{e}_{**}}$ are multiples of the following rational slopes, with the same multiplicative constant, as unoriented multi-slopes:
\begin{itemize}
\item $\Sigma\cap T_{e_*}$ and $\Sigma\cap T_{e_{**}}$: 
\begin{align}\label{5.5}
(|r|\lambda,\lambda'-\epsilon_r p\lambda-\delta|r|),
\end{align}
\item $\Sigma\cap T_{\bar{e}_*}$ and $\Sigma\cap T_{\bar{e}_{**}}$: 
\begin{align}\label{5.6}
(|r|\lambda',\lambda+\epsilon_r s\lambda'-\delta'|r|).
\end{align}
\end{itemize}
Here $\epsilon_r\in \{\pm1\}$ is the sign of $r$.
\end{enumerate}
\end{lemma}

A picture of $N=M\setminus\setminus (T_{|e_*|}\cup T_{|e_{**}|})$ and the pasting maps $f_{e_*},f_{e_{**}}$ is shown in Figure \ref{Figure1}.

\begin{figure}[htbp]
    \centering
    \def\svgwidth{4in} 
   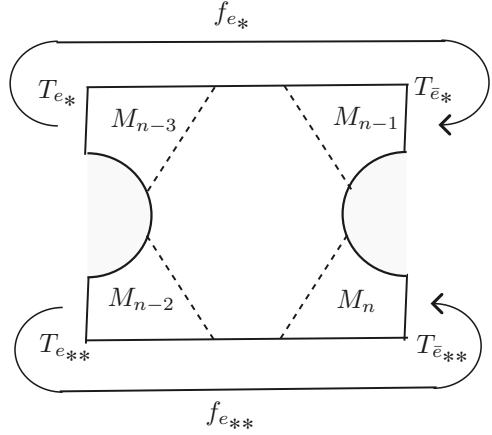
    \caption{A picture of $N=M\setminus\setminus (T_{|e_*|}\cup T_{|e_{**}|})$.}
     \label{Figure1}
\end{figure}

Note that the surface $\Sigma$ may not be orientable.
Apparently, there exists a more complicated version of Lemma \ref{horizontaladvanced} that considers more than two JSJ tori and drops some artificial assumptions. However, the current version is sufficient for our proof of Theorem \ref{lex}, and the notations and computations here are simpler.

\begin{proof}

We take a large integer $N$ as in the proof of Proposition \ref{horizontalbasic} such that $N\delta, N\delta'$ are also integers, and we still take $k_e=\frac{N}{|r_e|}$. For any oriented edge $e:v_i\to v_j$ such that $|e|\ne |e_*|,|e_{**}|$, we still equip $T_e\subset \partial M_i$ with the oriented rational slope $(|r_e|\lambda_i, \lambda_j-\epsilon_ep_e \lambda_i)$ as in equation \eqref{5.1}. 

For any $1\leq i\leq n-4$, by equation \eqref{5.4}, the $i$-th coordinate of $(Y-Z)\cdot \Lambda$ equals $0$. The computation in Lemma \ref{horizontalbasic} implies that the chosen oriented multi-slopes on $\partial M_i$ bound a connected, orientable, horizontal subsurface $\Sigma_i\subset M_i$.

For $i\in\{n-3,n-2,n-1,n\}$, the choice of oriented rational slopes in equation \eqref{5.1} may not give a horizontal subsurface in $M_i$, since $\delta$ and $\delta'$ may not be zero. We will modify the choice of oriented rational slopes on $T_{e_*}\subset \partial M_{n-3},T_{e_{**}}\subset \partial M_{n-2},T_{\bar{e}_*}\subset \partial M_{n-1},T_{\bar{e}_{**}}\subset \partial M_n$ to get a connected, orientable, horizontal subsurface $\Sigma_i\subset M_i$.

The computation is routine but tedious, so we will only give the detailed computation for $T_{e_*}\subset \partial M_{n-3}$. The other three computations are essentially the same.

Recall that we have $e_*:v_{n-3}\to v_{n-1}$. On $T_{e_*}$, we take the oriented rational slope 
$$(|r_{e_*}|\lambda_{n-3},\lambda_{n-1}-\epsilon_{e_*}p_{e_*}\lambda_{n-3}-\delta_{n-3}|r_{e_*}|)=(|r|\lambda,\lambda'-\epsilon_r p\lambda-\delta|r|),$$
and choose $k_{e_{*}}=\frac{N}{|r|}$.
Here $\delta_{n-3}$ denotes the $(n-3)$-th coordinate of the right-hand side of equation \eqref{5.4}. Note that if $\delta=0$, this rational slope is the same as the one in equation \eqref{5.1}. For $M_{n-3}$, the conditions in Lemma \ref{horizontalsurface} (1) and (2) clearly hold. To check Lemma \ref{horizontalsurface} (3), we compute
\begin{align*}
& \sum_{\{(e,j)|e:v_{n-3}\to v_j\}\setminus \{(e_*,n-1)\}}\frac{\lambda_j-\epsilon_ep_e \lambda_{n-3}}{|r_e|\lambda_{n-3}}+\frac{\lambda_{n-1}-\epsilon_{e_*}p_{e_*}\lambda_{n-3}-\delta|r_{e_*}|}{|r_{e_*}|\lambda_{n-3}}\\
=&\sum_{\{(e,j)|e:v_{n-3}\to v_j\}\setminus \{(e_*,n-1)\}}\Big(\frac{\lambda_j}{|r_e|\lambda_{n-3}}-\frac{p_e}{r_e}\Big)+\Big(\frac{\lambda_{n-1}}{|r_{e_*}|\lambda_{n-3}}-\frac{p_{e_*}}{r_{e_*}}\Big)-\frac{\delta}{\lambda_{n-3}}\\
=&\sum_{\{(e,j)|e:v_{n-3}\to v_j\}}\Big(\frac{\lambda_j}{|r_e|\lambda_{n-3}}-\frac{p_e}{r_e}\Big)-\frac{\delta}{\lambda_{n-3}}=0.
\end{align*}
The last equation holds since the first and second terms on the left side equal the $(n-3)$-th coordinates of the left and right hand sides of equation \eqref{5.4}, respectively, up to multiplying $\frac{1}{\lambda_{n-3}}$. So we get a connected, orientable, horizontal subsurface $\Sigma_{n-3}\subset M_{n-3}$ with given boundary slopes.

We choose oriented rational slopes for $e=e_{**},\bar{e}_{*},\bar{e}_{**}$ as the following, and take $k_{e_{**}}=k_{\bar{e}_{*}}=k_{\bar{e}_{**}}=\frac{N}{|r|}$. Then the same computation as above gives connected, orientable, horizontal subsurfaces $\Sigma_{n-2}\subset M_{n-2}, \Sigma_{n-1}\subset M_{n-1},\Sigma_n\subset M_n$.
\begin{itemize}
\item $e_{**}: (|r_{e_{**}}|\lambda_{n-2},\lambda_n-\epsilon_{e_{**}}p_{e_{**}}\lambda_{n-2}-\delta_{n-2}|r_{e_{**}}|)=-(|r|\lambda,\lambda'-\epsilon_r p\lambda-\delta|r|)$,
\item $\bar{e}_*: (|r_{\bar{e}_{*}}|\lambda_{n-1},\lambda_{n-3}-\epsilon_{\bar{e}_{*}}p_{\bar{e}_{*}}\lambda_{n-1}-\delta_{n-1}|r_{\bar{e}_{*}}|)=(|r|\lambda',\lambda+\epsilon_r s\lambda'-\delta'|r|)$,
\item $\bar{e}_{**}: (|r_{\bar{e}_{**}}|\lambda_{n},\lambda_{n-2}-\epsilon_{\bar{e}_{**}}p_{\bar{e}_{**}}\lambda_{n}-\delta_{n}|r_{\bar{e}_{**}}|)=-(|r|\lambda',\lambda+\epsilon_r s\lambda'-\delta'|r|)$.
\end{itemize}
These rational slopes are equal to the ones in item (5), as unoriented rational slopes.

Now we have a collection of connected, horizontal subsurfaces $\Sigma_i\subset M_i$. For any $e:v_i\to v_j$ such that $|e|\ne |e_*|,|e_{**}|$, we know that $\Sigma_i\cap T_e$ and $\Sigma_j\cap T_{\bar{e}}$ give the same unoriented multi-slope on $T_{|e|}$. So we can paste all $\Sigma_i$ along $T_{|e|}$ with $|e|\ne |e_*|,|e_{**}|$ to get the desired compact surface $\Sigma$ and the map $f:\Sigma\to M$. 

Since $\Sigma_{n-3}\cap T_{e_*}$ and $\Sigma_{n-1}\cap T_{\bar{e}_*}$ may not agree on $T_{|e_*|}$, while $\Sigma_{n-2}\cap T_{e_{**}}$ and $\Sigma_n\cap T_{\bar{e}_{**}}$ may not agree on $T_{|e_{**}|}$, we can not close up $\Sigma$ on $T_{|e_*|}$ or $T_{|e_{**}|}$. So $\Sigma$ has non-empty boundary, which is exactly $f^{-1}(T_{|e_*|}\cup T_{|e_{**}|})$.

Then the map $f:\Sigma\to M$ induces a proper embedding $\Sigma\to N=M\setminus \setminus (T_{|e_*|}\cup T_{|e_{**}|})$, and it satisfies all the required conditions.
\end{proof}

The behavior of matrices $Y$ and $Z$ is not nice when taking finite covers of graph $3$-manifolds, so we want to do the following change of variable. 

For each vertex manifold $M_i\subset M$ with $M_i=F_i\times S^1$, we use $\chi_i $ to denote $-\chi(F_i)$, which is a positive integer. Let $b_{ij}=\frac{y_{ij}}{\chi_i\chi_j}\text{\ and\ }w_i=\frac{z_i}{\chi_i^2}$. Then we define two new symmetric $n\times n$ rational matrices by
$B=(b_{ij})_{n\times n}$ and $W=\text{diag}(w_1,\cdots,w_n).$ Let $D=\text{diag}(\frac{1}{\chi_1},\cdots,\frac{1}{\chi_n})$, then we have 
\begin{align}\label{5.7}
B=DYD\text{\ and \ } W=DZD.
\end{align}

We now restate Lemma \ref{horizontaladvanced} in terms of matrices $B$ and $W$.

\begin{lemma}\label{horizontaladvanced2}
Let $M$ be a closed framed graph $3$-manifold as in Lemma \ref{horizontaladvanced}, such that $\chi_{n-3}=\chi_{n-2}=\chi$ and $\chi_{n-1}=\chi_n=\chi'$. If there exists a totally non-zero rational vector $\Lambda=(\lambda_1,\cdots, \lambda_n)^t$ such that 
\begin{align}\label{5.8}
(B-W)\cdot \Lambda=
(0, \cdots, 0, \delta, -\delta, \delta', -\delta')^t
\end{align}
with $\lambda_{n-3}=-\lambda_{n-2}=\lambda$, $\lambda_{n-1}=-\lambda_n=\lambda'$, and $\delta,\delta'\in \mathbb{Q}$. Then there exists a properly embedded, horizontal subsurface $\Sigma\to N=M\setminus\setminus (T_{|e_*|}\cup T_{|e_{**}|})$ as in Proposition \ref{horizontaladvanced}, with item (5) replaced by: 
\begin{itemize}
\item $\Sigma\cap T_{e_*}$ and $\Sigma\cap T_{e_{**}}$ are multiples of the following rational slopes, as unoriented multi-slopes:
\begin{align}\label{5.9}
(\frac{|r|\lambda}{\chi},\frac{\lambda'}{\chi'}-\epsilon_r p\frac{\lambda}{\chi}-\chi \delta|r|),
\end{align}
\item $\Sigma\cap T_{\bar{e}_*}$ and $\Sigma\cap T_{\bar{e}_{**}}$ are multiples of the following rational slopes, as unoriented multi-slopes:
\begin{align}\label{5.10}
(\frac{|r|\lambda'}{\chi'},\frac{\lambda}{\chi}+\epsilon_r s\frac{\lambda'}{\chi'}-\chi'\delta'|r|),
\end{align}
\item all four multiplicative constants are the same.
\end{itemize}

Moreover, if $N$ is connected, then $\Sigma$ is connected.
\end{lemma}

Again, $\Sigma$ may not be an orientable surface.

\begin{proof}
Let $\Delta=
(0, \cdots, 0, \delta, -\delta, \delta', -\delta')^t$. By equation \eqref{5.7}, equation \eqref{5.8} is equivalent to 
$$D(Y-Z)D\Lambda=\Delta \text{\ and\ }(Y-Z)(D\Lambda)=D^{-1}\Delta.$$

By the assumption on $\chi_i$ for $i=n-3,\cdots,n$, the last four coordinates of $D\Lambda$ are $\frac{\lambda}{\chi},-\frac{\lambda}{\chi},\frac{\lambda'}{\chi'},-\frac{\lambda'}{\chi'}$, and the last four coordinates of $D^{-1}\Delta$ are $\chi\delta,-\chi\delta,\chi'\delta',-\chi'\delta'$.

Then $\hat{\Lambda}=D\Lambda$ and $\hat{\Delta}=D^{-1}\Delta$ form a solution of $(Y-Z)\hat{\Lambda}=\hat{\Delta}$ satisfying the assumption of Lemma \ref{horizontaladvanced}. So item (5) of Lemma \ref{horizontaladvanced} gives the claimed slopes of $\Sigma\cap T_{e_*},\Sigma\cap T_{e_{**}},\Sigma\cap T_{\bar{e}_*},\Sigma\cap T_{\bar{e}_{**}}$.

For the moreover part, since $\Sigma$ intersects with each vertex manifold of $N$ at a connected, horizontal subsurface, the connectedness of $N$ implies the connectedness of $\Sigma$.
\end{proof}

\subsection{Behavior of equations under finite covers of graph $3$-manifolds}\label{finitecover}

Since we want to study finite covers of graph $3$-manifolds, we need to study the behavior of the left-hand side of equation \eqref{5.8} under finite covers. Results in this section essentially belong to \cite{WY}.

We will not study general finite covers between graph $3$-manifolds, but only study some special finite covers constructed in \cite{WY}. The following lemma is a more detailed version of \cite[Lemma 2.4]{WY}.

\begin{lemma}\label{cover}
Let $M$ be a closed framed graph $3$-manifold, with associated $n\times n$ matrices $B=(b_{i,j})_{n\times n}$ and $W=\text{diag}(w_1,\cdots,w_n)$. For any $x_{ij}\in \mathbb{Q}\cap (0,2)$ with $i,j=1,\cdots,n$ such that $x_{ij}=x_{ji}$ holds, there exists a framed graph $3$-manifold $\tilde{M}$, a frame preserving degree-$d$ cover $\phi:\tilde{M}\to M$, and  an order-$2$ deck transformation $\tau:\tilde{M}\to \tilde{M}$ (so $d$ is even) such that the following hold.
\begin{enumerate}
\item For any JSJ torus $\tilde{T}\subset \tilde{M}$, the restriction $\phi|_{\tilde{T}}:\tilde{T}\to M$ is a homeomorphism to a JSJ torus of $M$.
\item For any vertex manifold $M_i\subset M$, $\phi^{-1}(M_i)=\tilde{M}_{2i-1}\cup \tilde{M}_{2i}$ has two components and we label them by $2i-1$ and $2i$. Moreover, the deck transformation $\tau$ swaps $\tilde{M}_{2i-1}$ and $\tilde{M}_{2i}$.
\item If $M_i$ and $M_j$ are two vertex manifolds of $M$ intersecting on a JSJ torus $T\subset M$, then any component of $\phi^{-1}(M_i)$ intersects with any component of $\phi^{-1}(M_j)$ on at least two components of $\phi^{-1}(T)$.
\item Let $X_{ij}=
\begin{pmatrix}
x_{ij} & 2-x_{ij}\\
2-x_{ij} & x_{ij}
\end{pmatrix},$ then the $2n\times 2n$ matrices $\tilde{B}$ and $\tilde{W}$ associated to $\tilde{M}$ satisfy:
\begin{align}\label{5.11}
d\tilde{B}=(b_{ij}X_{ij})_{n\times n},\ d\tilde{W}=\text{diag}(2w_1,2w_1,2w_2,2w_2,\cdots,2w_n,2w_n).
\end{align}
\end{enumerate} 
Here $(b_{ij}X_{ij})_{n\times n}$ denotes the $2n\times 2n$ blocked matrix such that the $2\times 2$ block at the $i$-th row and $j$-th column is $b_{ij}X_{ij}$, for all $i,j=1,\cdots,n$.

Moreover, we can make $d$ as large as we want.
\end{lemma}

\cite[Lemma 2.4]{WY} only states that item (4) holds. Items (1), the first half of item (2), and the fact that $\phi:\tilde{M}\to M$ is frame preserving directly follow from the construction in \cite[Lemma 2.4]{WY}. Indeed, for any vertex manifold $M_v=F_v\times S^1\subset M$, they take a degree-$\frac{d}{2}$ cover $\tilde{F}_v\to F_v$ that restricts to each boundary component of $\tilde{F}_v$ as a homemorphism to its image, take two copies of $\tilde{F}_v\times S^1$, and paste them together to get $\tilde{M}$. The existence of $\tau$ and its property in item (2) also follows from the construction in \cite[Lemma 2.4]{WY}. Indeed, the covering map $\tilde{M}\to M$ factors through an intermediate cover $M'$ such that $M'$ is obtained by pasting only one copy of $\tilde{F}_v\times S^1$ together, so $\tilde{M}\to M'$ has degree $2$ and the desired deck transformation $\tau$ exists. Here we can make $d$ as large as we want since the proof of \cite[Lemma 2.4]{WY} only requires that $\frac{d}{2}$ is a multiple of the denominator of $\frac{x_{ij}}{2}$ for all $x_{ij}$.

For item (3), item (1) implies that $\phi^{-1}(T)$ has $d$ components. The proof of \cite[Lemma 2.4]{WY} implies that $\tilde{M}_{2i-1}\cap \tilde{M}_{2j-1}$, $\tilde{M}_{2i-1}\cap \tilde{M}_{2j}$, $\tilde{M}_{2i}\cap \tilde{M}_{2j-1}$, $\tilde{M}_{2i}\cap \tilde{M}_{2j}$ contain exactly $\frac{x_{ij}}{4}d$, $(\frac{1}{2}-\frac{x_{ij}}{4})d$, $(\frac{1}{2}-\frac{x_{ij}}{4})d$, $\frac{x_{ij}}{4}d$ components of $\phi^{-1}(T)$, respectively. Then item (3) follows from the assumption $x_{ij}\in (0,2)\cap \mathbb{Q}$ and $d$ can be arbitrarily large, while \cite[Lemma 2.4]{WY} only assumes $x_{ij}\in [0,2]\cap \mathbb{Q}$. 

By items (1) and (2), $\tilde{M}_{2i-1}\to M_{i}$ and  $\tilde{M}_{2i}\to M_{i}$ are both degree-$\frac{d}{2}$ covers that restrict to homeomorphisms on $S^1$-fibers, so the following holds for all $i=1,\cdots,n$.
\begin{align}\label{5.12}
\tilde{\chi}_{2i-1}=\tilde{\chi}_{2i}=\frac{d\chi_i}{2}.
\end{align}
Here, $\tilde{\chi}_i$ is the negative of the Euler characteristic of the base surface of $\tilde{M}_i$.

The following lemma follows from \cite[Lemma 3.1]{WY}.

\begin{lemma}\label{existsolution}
Let $C$ be a symmetric $n\times n$ matrix with rational entries such that for any $i,j=1,\cdots,n$, either $c_{ij}=b_{ij}=0$ or $|c_{ij}|< b_{ij}$ holds. Then $M$ has a degree-$d$ cover $\phi:\tilde{M}\to M$ (depending only on $C$) as in Lemma \ref{cover},  such that the following hold. For any totally non-zero rational vector $\Lambda=(\lambda_1,\cdots,\lambda_n)^t$ such that 
\begin{align}\label{5.13}
(C-W)\Lambda=
(0, \cdots, 0, \delta, \delta')^t,
\end{align}
the vector
$\tilde{\Lambda}=(\lambda_1,-\lambda_1,\cdots,\lambda_n,-\lambda_n)^t$ satisfies
\begin{align}\label{5.14}
d(\tilde{B}-\tilde{W})\tilde{\Lambda}=
(0, \cdots, 0, 2\delta, -2\delta, 2\delta', -2\delta')^t.
\end{align}

Moreover, we can make $d$ as large as we want.
\end{lemma}

If $b_{ij}=0$, we take $x_{ij}=1$; if $b_{ij}>0$, we take $x_{ij}=1+\frac{c_{ij}}{b_{ij}}\in \mathbb{Q}\cap (0,2)$. We apply Lemma \ref{cover} to this choice of $x_{ij}$ to construct $\tilde{M}$, then the result follows from a direct computation, as in the proof of \cite[Lemma 3.1]{WY}.

\section{A technical result on graph $3$-manifolds}\label{technicalsection}

The main result of this section is Proposition \ref{technicalresult}, which uses all the results in Section \ref{moregraph}. It constructs an interesting finite cover of a non-virtually fibered graph $3$-manifold.

The statement of Proposition \ref{technicalresult} is complicated, and it is based on the following wishful statement. Any closed graph $3$-manifold $M$ has a finite cover $\tilde{M}$ with a JSJ torus $T_{|e|}\subset \tilde{M}$, such that $N=\tilde{M}\setminus\setminus T_{|e|}$ has many horizontal subsurfaces that intersect with each of $T_e,T_{\bar{e}}\subset \partial N$ along all but finitely many unoriented slopes up to multiplicity.  Unfortunately, we cannot prove this wishful statement, but we can prove another version of it that cuts along two JSJ tori of $\tilde{M}$.

\begin{proposition}\label{technicalresult}
Let $M$ be a closed framed graph $3$-manifold that is not virtually fibered. Then there exists a framed graph $3$-manifold $\tilde{M}$, a frame preserving finite cover $\tilde{M}\to M$ with $2n$ vertex manifolds, and the following objects in $\tilde{M}$, such that the statement in the next paragraph holds.
\begin{enumerate}
\item Two oriented edges $e_*:v_{2n-3}\to v_{2n-1}$, $e_{**}:v_{2n-2}\to v_{2n}$ in the dual graph $\Gamma_{\tilde{M}}$ that correspond to JSJ tori $T_{|e_*|},T_{|e_{**}|}\subset \tilde{M}$, such that $\tilde{M}_{2n-1}$ and $\tilde{M}_{2n}$ have non-zero charges and $N=\tilde{M}\setminus\setminus (T_{|e_*|}\cup T_{|e_{**}|})$ is connected.
\item An order-$2$ deck transformation $\tau:\tilde{M}\to \tilde{M}$ that swaps $T_{|e_*|}$ and $T_{|e_{**}|}$, $\tilde{M}_{2n-3}$ and $\tilde{M}_{2n-2}$, $\tilde{M}_{2n-1}$ and $\tilde{M}_{2n}$, respectively.
\item An isomorphism $\iota:H_1(T_{e_*};\mathbb{Q})\to H_1(T_{\bar{e}_*};\mathbb{Q})$.
\end{enumerate}

For $N=\tilde{M}\setminus\setminus (T_{|e_*|}\cup T_{|e_{**}|})$, there are two oriented rational slopes $c_1^*,c_2^*$ on $T_{e_*}$, such that for any oriented rational slope $c$ on $T_{e_*}$ that is not a multiple of $c_1^*,c_2^*$, there exists a connected, properly embedded, horizontal subsurface $\Sigma_c\to N$ such that the following hold.
\begin{enumerate}
\item[(i)] $\Sigma_c$ intersects with each vertex manifold of $N$ at a connected, orientable, horizontal subsurface.
\item[(ii)] $\Sigma_c$ intersects with $T_{e_*}$ and $T_{e_{**}}$ along multiples of $c$ and $\tau(c)$, respectively, as unoriented multi-slopes.
\item[(iii)] $\Sigma_c$ intersects with $T_{\bar{e}_*}$ and $T_{\bar{e}_{**}}$ along multiples of $\iota(c)$ and $\tau(\iota(c))$, respectively, as unoriented multi-slopes.
\item[(iv)] The four multiples in items (ii) and (iii) are equal to each other.
\item[(v)] After composing with the pasting homomorphism $(f_{\bar{e}_*})_{\#}:H_1(T_{\bar{e}_*};\mathbb{Q})\to H_1(T_{e_*};\mathbb{Q})$, the composition $(f_{\bar{e}_*})_{\#}\circ\iota: H_1(T_{e_*};\mathbb{Q})\to H_1(T_{e_*};\mathbb{Q})$ has determinant $1$.
\end{enumerate}
\end{proposition}

Again, a picture of $N=\tilde{M}\setminus \setminus (T_{|e_*|}\cup T_{e_{**}})$ can be found in Figure \ref{Figure1}. 

Note that we do not claim that $\Sigma_c$ is orientable.

We assume that $c$ is an oriented slope for now. Since $\tau$ is a deck transformation, $c$ and $\tau(c)$ are both (primitive) slopes. However, since $\iota$ is only an isomorphism on homology with rational coefficients, $\iota(s)$ and $\tau(\iota(s))$ are only rational slopes, and the same rational multiple makes them slopes. So items (ii)-(iv) imply that $\Sigma_c\cap T_{e_*}$ and $\Sigma_c\cap T_{e_{**}}$ are both $k$ copies of some unoriened slopes, while $\Sigma_c\cap T_{\bar{e}_*}$ and $\Sigma_c\cap T_{\bar{e}_{**}}$ are both $k'$ copies of some unoriented slopes, but $k\ne k'$ in general.

We first prove a few lemmas.

\begin{lemma}\label{firstcover}
Let $M$ be a closed framed graph $3$-manifold that is not virtually fibered, then $M$ admits a frame preserving finite cover $\tilde{M}$ such that the following hold.
\begin{enumerate}
\item There exists a vertex manifold $\tilde{M}_v\subset \tilde{M}$ with non-zero charge such that $\tilde{M}\setminus\setminus \tilde{M}_v$ is connected.
\item $\tilde{M}$ has at least two vertex manifolds with non-zero charges.
\end{enumerate}
\end{lemma}

\begin{proof}
Since $M$ is not virtually fibered, Theorem \ref{chargeless} implies that $M$ has a vertex manifold $M_v$ with non-zero charge. Since we assume that $M$ satisfies the conclusion of Lemma \ref{simplegraph} (2), each vertex of the dual graph $\Gamma_M$ has valence at least $3$, so $\pi_1(\Gamma_M)$ is a non-abelian free group.

Let $N_1,\cdots, N_k$ be the vertex manifolds of $M$ adjacent to $M_v$, with $N_i=F_i\times S^1$, and let $N_1',\cdots, N_l'$ be the components of $M\setminus \setminus (M_v\cup (\cup_{i=1}^kM_i))$. By Lemma \ref{simplegraph} (2), the genus of $F_i$ is at least $1$, so there exists a connected double cover $\tilde{F}_i\to F_i$ that restricts to a homeomorphism on each boundary component of $\tilde{F}_i$ to its image.

We take one copy of $\tilde{N}_i=\tilde{F}_i\times S^1$ with $i=1,\cdots,k$, two copies of $N_j'$ with $j=1,\cdots,l$, and two copies of $M_v$. By using the pasting maps of $M$, we paste these pieces along their boundaries to get a double cover $p:\tilde{M}\to M$. We denote the two copies of $M_v$ in $\tilde{M}$ by $\tilde{M}_v$ and $\tilde{M}_v'$. Then $\tilde{M}\setminus\setminus \tilde{M}_v$ is connected because any point in it can be connected to $\tilde{M}_v'$. Both $\tilde{M}_v$ and $\tilde{M}_v'$ have the same non-zero charge, since both $p|_{\tilde{M}_v}:\tilde{M}_v\to M_v$ and $p|_{\tilde{M}_v'}:\tilde{M}_v'\to M_v$ are homeomorphisms that preserve $S^1$-fibers of adjacent vertex manifolds.
\end{proof}

We abuse notation and use $M$ to denote the resulting manifold of Lemma \ref{firstcover}, and use $M_v$ to denote the vertex manifold $\tilde{M}_v$. The next lemma constructs a maximal spanning subtree of $\Gamma_M$ and a labeling of the vertices of $\Gamma_M$.

\begin{lemma}\label{tree}
Let $M$ be a closed framed graph $3$-manifold with a vertex manifold $M_v\subset M$ satisfying the conclusion of Lemma \ref{firstcover}, and let $n$ be the number of vertex manifolds of $M$. Then there exists a maximal spanning subtree $T\subset \Gamma_M$ and a labeling of vertices of $\Gamma_M$ by $\{1,\cdots,n\}$ such that the following hold.
\begin{enumerate}
\item The vertex of $T$ corresponding to $M_v$ is labeled by $n$, and we denote $v$ and $M_v$ by $v_n$ and $M_n$, respectively.
\item In the tree $T$, $v_n$ has valence $1$, and the only vertex of $T$ adjacent to $v_n$ is $v_{n-1}$.
\item For any $i\in \{1,\cdots,n\}$, let $T_i$ be the subtree of $T$ spanned by vertices $v_k$ (including $v_i$) such that the unique reduced path from $v_n$ to $v_k$ passes through $v_i$. Then all vertices of $T_i$ have labelings in $[1,i]\cap \mathbb{Z}$.
\end{enumerate}
\end{lemma}

A picture of the tree $T$ is shown in Figure \ref{Figure2}, which shows the distinguished vertices $v_n,v_{n-1}$, a few vertices $v_j,v_i,v_k$, and the corresponding subtree $T_i$ as defined in Lemma \ref{tree} (3). In this picture, $k<i<j$ holds.

\begin{figure}[htbp]
    \centering
    \def\svgwidth{4in} 
   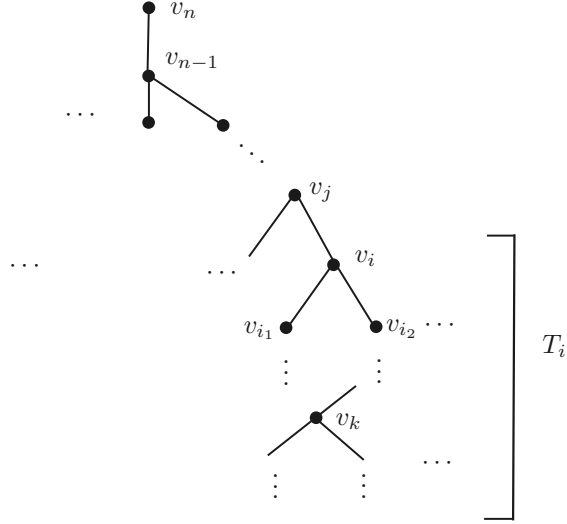
    \caption{A picture of the tree $T$.}
     \label{Figure2}
\end{figure}

\begin{proof}
We label the vertex in $\Gamma_M$ corresponding to $M_v$ by $n$. So we denote $v$ by $v_n$, and denote $M_v$ by $M_n$. Let $\Gamma'$ be the subgraph of $\Gamma_M$ obtained by deleting $v_n$ and all edges adjacent to $v_n$.  By Lemma \ref{firstcover} (1), $M\setminus\setminus M_n$ is connected, so $\Gamma'$ is connected. 

Let $T'$ be a maximal spanning tree of $\Gamma'$, and let $|e|$ be one edge between $T'$ and $v_n$. The desired spanning tree $T$ is the union of $T'$, $|e|$, and $v_n$. We label the end point of $|e|$ in $T'$ by $n-1$, so item (2) holds. 


It remains to label the vertices of $T\setminus \{v_{n-1},v_n\}$ such that item (3) holds, which is quite flexible. For example, suppose there are $k_1$ vertices of $T\setminus \{v_{n-1},v_n\}$ with distance $1$ to $v_{n-1}$, then we label them by $n-2,n-3,\cdots,n-k_1-1$ in an arbitrary way. Then we label the $k_2$ vertices of $T\setminus  \{v_{n-1},v_n\}$ with distance $2$ to $v_{n-1}$ by $n-k_1-2,n-k_1-3,\cdots,n-k_1-k_2-1$ in an arbitrary way. We continue this process to label all vertices of $T$, then item (3) holds.
\end{proof}

The next lemma is a more detailed version of \cite[Lemma 3.2]{WY}.

\begin{lemma}\label{twosolutions}
Let $M$ be a non-virtually fibered, closed framed graph $3$-manifold satisfying the conclusion of Lemma \ref{firstcover}, and we equip $\Gamma_M$ with a spanning tree $T$ and a labeling of vertices satisfying Lemma \ref{tree}. Let $B=(b_{ij})_{n\times n}$ and $W=\text{diag}(w_1,\cdots,w_n)$ be the $n\times n$ matrices associated to $M$ and the given labeling. Then there exists a symmetric rational $n\times n$ matrix $C=(c_{ij})_{n\times n}$ such that the following hold.
\begin{enumerate}
\item For any $i,j=1,\cdots,n$, either $c_{ij}=b_{ij}=0$ or $|c_{ij}|<b_{ij}$ holds.
\item $c_{n-1,n}\ne 0$.
\item There are two totally non-zero rational vectors 
\begin{align}\label{6.1}
\Lambda=(P,P^2,\cdots,P^{n-1},\frac{c_{n-1,n}P^{n-1}}{w_n})^t,\ \Lambda'=(P,P^2,\cdots,P^{n-1},P^n)^t
\end{align} 
such that 
\begin{align}\label{6.2}
(C-W)\Lambda=(0,\cdots,0,\delta,0)^t,\ (C-W)\Lambda'=(0,\cdots,0,0,\delta')^t.
\end{align}
Here, $P$ is a positive integer, and we have
\begin{align}\label{6.3}
\delta=(c_{n-1,n}-Pw_n)\frac{c_{n-1,n}P^{n-1}}{w_n},\delta'=(c_{n-1,n}-Pw_n)P^{n-1}\in \mathbb{Q}\setminus \{0\}.
\end{align}
\end{enumerate}
\end{lemma}

Note that $w_n\ne 0$ by Lemma \ref{firstcover} (1) and Lemma \ref{tree} (1), so equations \eqref{6.1} and \eqref{6.3} have no zeros in denominators.

\begin{proof}
The construction of $C$ and $\Lambda'$ is the same as the construction in \cite[Lemma 3.2]{WY}, but we will repeat it since the details will be used to check item (2) and to construct $\Lambda$.

Recall that all entries of $B$ are non-negative, and $b_{ij}>0$ if and only if there is an edge between $v_i$ and $v_j$ in $\Gamma_M$. Since $M_n$ has non-zero charge, we have $w_n=\frac{z_n}{\chi_n^2}\ne 0$. We define the following two positive rational numbers 
\begin{align}\label{6.4}
b=\min\{b_{ij}\ |\ b_{ij}>0\},\ w=\max\{|w_i|\}.
\end{align}
We take a large positive integer 
\begin{align}\label{6.5}
P>\frac{w}{b}+2
\end{align} to be determined later.

For the matrix $C=(c_{ij})_{n\times n}$, we define $c_{ij}= 0$ if there is no edge in the spanning tree $T$ between $v_i$ and $v_j$. Item (1) apparently holds for such a $c_{ij}$. We will define the other $c_{ij}$ inductively by $i$, by using the equation 
\begin{align}\label{6.6}
(C-W)(P,P^2,\cdots,P^n)^t=(0,\cdots,0,0,\delta')^t
\end{align}
deduced from the $\Lambda'$ part of equations \eqref{6.1} and \eqref{6.2}. Actually, we will prove that $|c_{ij}|<b$ always holds.

We start with $i=1$. By Lemma \ref{tree} (3), $v_1$ is a valence-$1$ vertex of $T$, so it is adjacent to a unique vertex $v_j$ in $T$ with $j\geq 2$. So $c_{1k}=0$ for any $k\ne j$ and $b_{1j}>0$ holds. The first coordinate of \eqref{6.6} gives equation $c_{1j}P^j-w_1P=0$, so we have to define
\begin{align}\label{6.7}
c_{1j}=\frac{w_1}{P^{j-1}},
\end{align} 
and $|c_{1j}|\leq \frac{w}{P}<b$ holds.

Suppose we have defined the first $i-1$ rows of $C$ for $i\leq n-1$. Then we want to define the $i$-th row of $C$. Since $C$ is symmetric, the first $i-1$ columns of $C$ are also defined. For the vertex $v_i$ in $\Gamma_M$, by Lemma \ref{tree} (3), there is a unique vertex $v_j$ adjacent to $v_i$ in $T$ such that $j>i$. So for the $i$-th row of $C$, it remains to define $c_{ij}$, and note that $b_{ij}>0$ holds. The $i$-th coordinate of \eqref{6.6} gives
$$c_{i1}P+c_{i2}P^2+\cdots +c_{i,i-1}P^{i-1}+c_{ij}P^j-w_iP^i=0,$$ so we have 
\begin{align}\label{6.8}
c_{ij}=\frac{w_i}{P^{j-i}}-\frac{1}{P^j}(c_{i1}P+c_{i2}P^2+\cdots +c_{i,i-1}P^{i-1}).
\end{align}
Then we estimate 
\begin{align*}
& |c_{ij}|\leq |\frac{w_i}{P^{j-i}}|+\frac{1}{P^j}|c_{i1}P+c_{i2}P^2+\cdots +c_{i,i-1}P^{i-1}|\\
\leq \ & \frac{w}{P}+\frac{b}{P^j}\frac{P^i-P}{P-1}\leq \frac{w+b}{P}< b.
\end{align*}
Here, the second inequality follows from $j>i$ and $P>2$, and the last inequality follows from \eqref{6.5}.

Once we define the first $n-1$ rows of $C$, we get the whole matrix $C$, since $C$ is symmetric and $c_{nn}=0$. Since all $w_i$ are rational numbers, all $c_{ij}$ are also rational numbers by equations \eqref{6.7} and \eqref{6.8}.  Then equation \eqref{6.6} holds since we do not require $\delta'$ to be zero. By Lemma \ref{tree} (2), the only non-zero entry in the $n$-th row of $C$ is $c_{n,n-1}=c_{n-1,n}$, so the last coordinate of equation  \eqref{6.6} gives $\delta'=(c_{n-1,n}-w_nP)P^{n-1}$. Since $M$ is not virtually fibered, Lemma \ref{horizontalbasic} implies that $\delta'\ne 0$.

Up until now, our proof is identical to the proof of \cite[Lemma 3.2]{WY}. In the following, we will prove item (2) and the part of item (3) on $\Lambda$ and $\delta$. 

We first assume that item (2) holds (i.e. $c_{n-1,n}\ne 0$) and prove item (3). 

Since $c_{n-1,n}\ne 0$, the vector $\Lambda$ in \eqref{6.1} is a totally non-zero rational vector. 
Since the only non-zero entry in the last column of $C$ is $c_{n-1,n}$, while $\Lambda$ and $\Lambda'$ agree except on the last coordinate, the first $n-2$ coordinates of $(C-W)\Lambda$ equal the first $n-2$ coordinates of $(C-W)\Lambda'$, which are all zero.
Since the only non-zero entry in the last row of $C$ is $c_{n,n-1}=c_{n-1,n}$, the last coordinate of $(C-W)\Lambda$ is $c_{n-1,n}P^{n-1}-w_n\frac{c_{n-1,n}P^{n-1}}{w_n}=0$. By taking the difference of two equations of \eqref{6.2}, we have
$$\delta=(n-1)\text{-th\ coordinate\ of\ }(C-W)(\Lambda-\Lambda')=c_{n-1,n}(\frac{c_{n-1,n}P^{n-1}}{w_n}-P^n),$$ 
which gives the first equation in \eqref{6.3}.
By Lemma \ref{horizontalbasic} again, we have $\delta\ne 0$ since $M$ is not virtually fibered. This proves the first equation in \eqref{6.1}.

It remains to choose a positive integer $P$ satisfying equation \eqref{6.5} so that item (2) holds (i.e. $c_{n-1,n}\ne 0$). By equations \eqref{6.7} and \eqref{6.8}, each $c_{ij}$ is a polynomial of $\frac{1}{P}$, whose coefficients are linear combinations of $w_k$. We need to figure out this polynomial more explicitly.

For any vertex $v_j$ of $T$, it gives a subtree $T_j\subset T$ as in Lemma \ref{tree} (3), then any vertex $v_k$ of $T_j$ satisfies $k\leq j$. For any vertex $v_k\in T_j\setminus \{v_j\}$, we define a positive integer $d_{kj}$ as follows. If $v_k$ is adjacent to $v_j$, we define 
\begin{align}\label{6.9}
d_{kj}=j-k.
\end{align}
If $v_k$ is not adjacent to $v_j$, then there exists a unique vertex $v_{j'}$ in $T_j$ adjacent to $v_j$, such that $v_k\in T_{j'}$. Then we have $k<j'<j$ and we define 
\begin{align}\label{6.10}
d_{kj}=j+j'-2k.
\end{align}

{\bf Claim.} For any edge in $T$ between $v_i$ and $v_j$ with $j>i$, we have
\begin{align}\label{6.11}
c_{ij}=\sum_{v_k\in T_i}\epsilon_{kj}\frac{w_k}{P^{d_{kj}}},
\end{align}
with $\epsilon_{kj}\in \{\pm1\}$. A picture of $v_j,v_i,v_k$ and $T_i$ can be found in Figure \ref{Figure2}. We will prove this claim by induction on $i$. Note that for any $i\leq n-1$, there is a unique $j>i$ such that $T$ has an edge between $v_i$ and $v_j$.

For $c_{1j}$ as in \eqref{6.7}, we have $i=1$ and $T_{1}$ has a unique vertex $v_1$. So the only choice of $k$ in equation \eqref{6.11} is $k=1$, and we have $d_{1j}=j-1$. Then equation \eqref{6.7} implies that equation \eqref{6.11} holds for $c_{1j}$.

Now we assume that equation \eqref{6.11} holds for all $c_{kl}$ with $k<l$ and $k\leq i-1$. Then we consider the formula of $c_{ij}$ in \eqref{6.8}, and note that $T_i\subset T_j$ holds. Let $v_{i_1},\cdots,v_{i_m}$ be the vertices in $T$ adjacent to $v_i$ other than $v_j$. Then we have $i_1,\cdots,i_m<i$, and the only non-zero $c_{i\bullet}$'s on the right side of \eqref{6.8} are $c_{i,i_1},\cdots,c_{i,i_m}$. See Figure \ref{Figure2} for a picture of $v_j,v_i$ and $v_{i_1},\cdots,v_{i_m}$.

Note that the set of vertices of $T_i$ is the disjoint union of sets of vertices of $T_{i_1},\cdots,T_{i_m}$ and $\{v_i\}$. 
By the induction hypothesis \eqref{6.11}, $c_{i,i_t}=c_{i_t,i}$ only consists of terms contributed by vertices of $T_{i_t}$, for any $t=1,\cdots,m$. Since the trees $T_{i_1},\cdots,T_{i_m}$ are disjoint, for any $v_k\in T_i$ with $w_k\ne 0$, $w_k$ shows up exactly once on the right side of \eqref{6.8}. So we only need to figure out the power of $p$ in $\frac{w_k}{P^{\bullet}}$ on the right hand side of equation \eqref{6.8}. 

If $k=i$, then $v_k=v_i$ is adjacent to $v_j$ and $d_{kj}=d_{ij}=j-i$. The term $\frac{w_i}{P^{d_{ij}}}=\frac{w_i}{P^{j-i}}$ in \eqref{6.11} is exactly the first term of \eqref{6.8}. If $k\ne i$, then $v_k$ lies in a unique $T_{i_t}$. By the induction hypothesis, in $c_{ii_t}=c_{i_ti}$, the term corresponding to $v_k$ is $\frac{w_k}{P^{d_{ki}}}$, up to a sign. So the term in \eqref{6.8} corresponding to $v_k$ is $\frac{w_k}{P^{d_{ki}+(j-i_t)}},$ up to a sign. Comparing to equation \eqref{6.11}, it remains to check 
\begin{align}\label{6.12}
d_{kj}=d_{ki}+(j-i_t).
\end{align}
If $k=i_t$, then $d_{i_tj}=j+i-2i_t$ and $d_{i_ti}=i-i_t$, so \eqref{6.12} holds. If $k\ne i_t$, then $d_{kj}=j+i-2k$ and $d_{ki}=i+i_t-2k$, so \eqref{6.12} still holds. This finishes the proof of the Claim, and equation \eqref{6.11} holds. We can also figure out $\epsilon_{kj}$ in \eqref{6.11} if we work harder, but it is not helpful for our proof.

Now we return to the proof of item (2), and we want to choose a positive integer $P$ satisfying equation \eqref{6.5} so that $c_{n-1,n}\ne 0$. By equation \eqref{6.11}, we have 
\begin{align}\label{6.13}
c_{n-1,n}=\sum_{v_k\in T_{n-1}}\epsilon_{kn}\frac{w_k}{P^{d_{kn}}}=\sum_{k=1}^{n-1}\epsilon_{kn}\frac{w_k}{P^{d_{kn}}}.
\end{align}
Here, the second equation holds by Lemma \ref{tree} (2).

By Lemma \ref{firstcover} (2), there exists $i_0\in \{1,\cdots,n-1\}$ such that $w_{i_0}\ne 0$, and we can assume that  $i_0$ is the largest such integer. If $i_0=n-1$, then equation \eqref{6.9} gives $d_{n-1,n}=1$, and equation \eqref{6.10} gives $d_{kn}> n-k\geq 2$ for any $k\leq n-2$. So the only $\frac{1}{P}$ term in $c_{n-1,n}$ is $\pm\frac{w_{n-1}}{P}$. Thus $c_{n-1,n}$ is a non-trivial polynomial of $\frac{1}{P}$, and we can choose an integer $P$ satisfying equation \eqref{6.5} so that $c_{n-1,n}\ne 0$.

If $i_0\leq n-2$, then equation \eqref{6.10} gives $d_{i_0n}=2n-1-2i_0$. By our choice of $i_0$, any other $k\in \{1,\cdots,n-1\}$ with $w_k\ne 0$ satisfies $k< i_0$. So equation \eqref{6.10} gives $d_{kn}=2n-1-2k>2n-1-2i_0=d_{i_0n}$. So the only $\frac{1}{P^{d_{i_0n}}}$ term in $c_{n-1,n}$ is $\pm\frac{w_{i_0}}{P^{d_{i_0n}}}$. Again, we can choose an integer $P$ satisfying equation \eqref{6.5} so that $c_{n-1,n}\ne 0$.

This finishes the proof of this lemma.
\end{proof}

Now we are ready to prove Proposition \ref{technicalresult}.

\begin{proof}
We start with a closed framed graph $3$-manifold $M$ that is not virtually fibered. We apply Lemma \ref{firstcover} to construct a framed finite cover of $M$ and still denote the resulting manifold by $M$. Then Lemma \ref{tree} gives us a maximal spanning tree $T\subset \Gamma_M$ and a labeling of vertices of $M$ by $\{1,\cdots,n\}$, such that $M_{n-1}$ is adjacent to $M_n$ and $M_n$ has non-zero charge.

Let $B=(b_{ij})_{n\times n}$ and $W=\text{diag}(w_1,\cdots,w_n)$ be the $n\times n$ matrices associated to $M$. Then Lemma \ref{twosolutions} gives us an $n\times n$ matrix $C$ and two totally non-zero rational vectors $$\Lambda=(P,P^2,\cdots,P^{n-1},\frac{cP^{n-1}}{w})^t,\ \Lambda'=(P,P^2,\cdots,P^{n-1},P^n)^t$$ as in equation \eqref{6.1}, such that equations \eqref{6.2} and \eqref{6.3} hold. Here we use $c$ and $w$ to denote the non-zero rational numbers $c_{n-1,n}$ and $w$ in Lemma \ref{twosolutions}, respectively.

Then Lemmas \ref{cover} and \ref{existsolution} give us a frame preserving degree-$d$ cover $\phi:\tilde{M}\to M$ and an order-$2$ deck transform $\tau:\tilde{M}\to \tilde{M}$, where $d$ is a large even number to be determined later. By Lemma \ref{cover} (2), for each vertex manifold $M_i\subset M$, $\phi^{-1}(M_i)$ consists of two connected components $\tilde{M}_{2i-1}, \tilde{M}_{2i}$ swapped by $\tau$. Since $M_{n-1}$ and $M_n$ are adjacent in $M$, Lemma \ref{cover} (3) provides us a JSJ torus $T_{|e_*|}\subset \tilde{M}$ adjacent to $\tilde{M}_{2n-3}$ and $\tilde{M}_{2n-1}$, and we orient $e_*$ as $e_*:\tilde{v}_{2n-3}\to \tilde{v}_{2n-1}$. Here $\tilde{v}_i$ is the vertex in $\Gamma_{\tilde{M}}$ corresponding to $\tilde{M}_i$. Then $T_{|e_{**}|}=\tau(T_{|e_*|})$ is adjacent to $\tilde{M}_{2n-2}$ and $\tilde{M}_{2n}$, and we orient it as $e_{**}:\tilde{v}_{2n-2}\to \tilde{v}_{2n}$. Since $M_n$ has non-zero charge in $M$, Lemma \ref{cover} (4) implies that $\tilde{M}_{2n-1}$ and $\tilde{M}_{2n}$ have non-zero charges in $\tilde{M}$. Lemma \ref{cover} (3) also gives other JSJ tori in $\tilde{M}_{2n-3}\cap \tilde{M}_{2n-1}$ and $\tilde{M}_{2n-2}\cap \tilde{M}_{2n}$, so $N=\tilde{M}\setminus \setminus (T_{|e_*|}\cup T_{|e_{**}|})$ is connected. We finish the construction of the desired objects in items (1) and (2) of this proposition.

Now we check that the objects in items (1) and (2) satisfy the assumptions of Lemmas \ref{horizontaladvanced} and \ref{horizontaladvanced2}, so that we can apply Lemma  \ref{horizontaladvanced2} to construct horizontal subsurfaces in $N=\tilde{M}\setminus\setminus (T_{|e_*|}\cup T_{|e_{**}|})$. By Lemma \ref{cover} (1), $T_{|e_*|}$ and $T_{|e_{**}|}$ cover the same JSJ torus in $M$ by homeomorphism. Since $\tilde{M}\to M$ is frame preserving,  $f_{e_*}=f_{e_{**}}=
\begin{pmatrix}
p & r \\
q & s
\end{pmatrix}$ is independent of the cover $\tilde{M}$ . By \eqref{5.12}, we have 
\begin{align}\label{6.14}
\tilde{\chi}_{2n-3}=\tilde{\chi}_{2n-2}=\frac{d\chi_{n-1}}{2}=\frac{d\chi}{2} \text{\ and\ } \tilde{\chi}_{2n-1}=\tilde{\chi}_{2n}=\frac{d\chi_n}{2}=\frac{d\chi'}{2}.
\end{align}
Here we use $\chi$ and $\chi'$ to denote $\chi_{n-1}$ and $\chi_n$, respectively.

As in Lemma \ref{cover} (4), let $\tilde{B}$ and $\tilde{W}$ be the $2n\times 2n$ matrices associated to $\tilde{M}$. Then, for the following totally non-zero rational vectors 
$$
\tilde{\Lambda}=(P,-P,P^2,-P^2,\cdots,P^{n-1},-P^{n-1},\frac{cP^{n-1}}{w},-\frac{cP^{n-1}}{w})^t,
$$
$$
\tilde{\Lambda}'=(P,-P,P^2,-P^2,\cdots,P^{n-1},-P^{n-1},P^n,-P^n)^t,
$$
Lemma \ref{twosolutions} (3) and Lemma \ref{existsolution} imply that 
$$d(\tilde{B}-\tilde{W})\tilde{\Lambda}=(0,0,\cdots,0,0,2\delta,-2\delta,0,0)^t,$$
$$d(\tilde{B}-\tilde{W})\tilde{\Lambda}'=(0,0,\cdots,0,0,0,0,2\delta',-2\delta')^t,$$
for $\delta,\delta'\in \mathbb{Q}\setminus \{0\}$ as in equation \eqref{6.3}.

For any $\alpha,\beta\in \mathbb{Q}$, we have
\begin{align*}
&\alpha\tilde{\Lambda}+\beta\tilde{\Lambda}'=\\
&((\alpha+\beta)P,-(\alpha+\beta)P,\cdots, (\alpha+\beta)P^{n-1},-(\alpha+\beta)P^{n-1},(\alpha c+\beta Pw)\frac{P^{n-1}}{w},-(\alpha c+\beta Pw)\frac{P^{n-1}}{w})^t
\end{align*}
and
\begin{align}\label{6.15}
d(\tilde{B}-\tilde{W})(\alpha\tilde{\Lambda}+\beta\tilde{\Lambda}')=(0,0,\cdots,0,0,2\alpha\delta,-2\alpha\delta,2\beta\delta',-2\beta\delta')^t.
\end{align}
Here $\alpha\tilde{\Lambda}+\beta\tilde{\Lambda}'$ is a totally non-zero rational vector if and only if
\begin{align}\label{6.16}
\alpha+\beta\ne 0 \text{\ and\ } \alpha c+\beta Pw\ne 0.
\end{align}

For any $\alpha,\beta\in \mathbb{Q}$ satisfying \eqref{6.16}, we apply Lemma \ref{horizontaladvanced2} to $\alpha\tilde{\Lambda}+\beta\tilde{\Lambda}'$ to construct a properly embedded, horizontal subsurface $\Sigma_{\alpha,\beta}\to N=\tilde{M}\setminus\setminus (T_{|e_*|}\cup T_{|e_{**}|})$ such that the following hold:
\begin{enumerate}
\item[(a)] $\Sigma_{\alpha,\beta}$ intersects with each vertex manifold of $N$ at a connected, orientable, horizontal subsurface (Lemma \ref{horizontaladvanced} (3)).
\item[(b)] $\Sigma_{\alpha,\beta}\cap T_{e_*}$ and $\Sigma_{\alpha,\beta}\cap T_{e_{**}}$ are multiples of the following rational slope, as unoriented multi-slopes:
\begin{align}\label{6.17}
(\frac{2|r|(\alpha+\beta)P^{n-1}}{d\chi},
\frac{2(\alpha c+\beta Pw)P^{n-1}}{d\chi'w}-
\epsilon_r p\frac{2(\alpha+\beta)P^{n-1}}{d\chi}
-\alpha\chi\delta|r|)
\end{align}
\item[(c)] $\Sigma_{\alpha,\beta}\cap T_{\bar{e}_*}$ and $\Sigma_{\alpha,\beta}\cap T_{\bar{e}_{**}}$ are multiples of the following rational slope, as unoriented multi-slopes:
\begin{align}\label{6.18}
(\frac{2|r|(\alpha c+\beta Pw)P^{n-1}}{d\chi'w},
\frac{2(\alpha+\beta)P^{n-1}}{d\chi}+
\epsilon_r s\frac{2(\alpha c+\beta Pw)P^{n-1}}{d\chi'w}
-\beta\chi'\delta'|r|),
\end{align}
\item[(d)] all four multiplicative constants in items (b) and (c) are the same. 
\end{enumerate}
Moreover, since $N$ is connected, $\Sigma_{\alpha,\beta}$ is connected.

Now we have two homomorphisms $\eta:\mathbb{Q}^2\to H_1(T_{e_*};\mathbb{Q})$ and $\eta':\mathbb{Q}^2\to H_1(T_{\bar{e}_*};\mathbb{Q})$, where $\eta(\alpha,\beta)$ is define by equation \eqref{6.17} for any $(\alpha,\beta)\in \mathbb{Q}^2$, and $\eta'(\alpha,\beta)$ is defined by \eqref{6.18}. Here we equip $H_1(T_{e_*};\mathbb{Q})$ and $H_1(T_{\bar{e}_*};\mathbb{Q})$ with ordered basis given by the frame of $\tilde{M}$.

We rewrite \eqref{6.17} as 
$$
\begin{pmatrix}
\frac{2|r|P^{n-1}}{d\chi} & \frac{2|r|P^{n-1}}{d\chi}\\
\frac{2cP^{n-1}}{d\chi'w}-\frac{2\epsilon_r pP^{n-1}}{d\chi}-\chi\delta|r| & \frac{2P^n}{d\chi'}-\frac{2\epsilon_rpP^ {n-1}}{d\chi}
\end{pmatrix}
\begin{pmatrix}
\alpha \\
\beta
\end{pmatrix},
$$
and the $2\times 2$ matrix has determinant 
\begin{align}\label{6.19}
\frac{2|r|P^{n-1}}{d\chi}(\frac{2P^n}{d\chi'}-\frac{2cP^{n-1}}{d\chi'w}+\chi\delta|r|)=\frac{4|r|(wP-c)P^{2n-2}}{\chi\chi'w}\cdot \frac{1}{d^2}+2|r|^2\delta P^{n-1}\cdot \frac{1}{d}
\end{align}
Note that $2|r|^2\delta P^{n-1}\ne 0$ holds. Here, $r\ne 0$ since fibers in adjacent vertex manifolds do not match (equation \eqref{2.1}), $\delta\ne 0$ since $M$ is not virtually fibered (equation \eqref{6.3}), $P\ne 0$ since it is a positive integer (Lemma \ref{twosolutions}). So we can choose the even integer $d$ so that \eqref{6.19} is not zero, thus $\eta:\mathbb{Q}^2\to H_1(T_{e_*};\mathbb{Q})$ is an isomorphism.

Similarly, we can rewrite \eqref{6.18} as the matrix-vector form. Then a direct but more tedious computation gives the determinant of the $2\times 2$ matrix:
\begin{align}\label{6.20}
-\frac{4|r|(wP-c)P^{2n-2}}{\chi\chi'w}\cdot \frac{1}{d^2}-\frac{2|r|^2c \delta' P^{n-1}}{w}\cdot \frac{1}{d}
\end{align}
By equation \eqref{6.3}, we have $\delta=\frac{c}{w}\delta'$. So equations \eqref{6.19} and \eqref{6.20} are negative of each other, and $\eta':\mathbb{Q}^2\to H_1(T_{\bar{e}_*};\mathbb{Q})$ is also an isomorphism.

Then we define the isomorphism $\iota:H_1(T_{e_*};\mathbb{Q})\to H_1(T_{\bar{e}_*};\mathbb{Q})$ in item (3) by $\iota=\eta'\circ \eta^{-1}$. Equations \eqref{6.19} and \eqref{6.20} imply that $\iota$ has determinant $-1$ under the coordinates given by the frame of $\tilde{M}$. Since $f_{\bar{e}_*}$ equals
$(f_{e_*})^{-1}=\begin{pmatrix}
p & r \\
q & s
\end{pmatrix}^{-1}$ with determinant $-1$, $(f_{\bar{e}_*})_{\#}\circ \iota: H_1(T_{e_*};\mathbb{Q})\to H_1(T_{e_*};\mathbb{Q})$ has determinant $1$, thus item (v) holds.

The two exceptional rational slopes $c_1^*,c_2^*$ on $T_{e_*}$ are given by $\eta (\alpha,\beta)\in H_1(T_{e_*};\mathbb{Q})$ for some $(\alpha,\beta)\in \mathbb{Q}^2\setminus \{(0,0)\}$ such that $\alpha+\beta=0$ or $\alpha c+\beta Pw=0$, respectively (equation \eqref{6.16}). For any rational slope $c$ on $T_{e_*}$ that is not a multiple of  $c_1^*$ or $c_2^*$, there exists $(\alpha,\beta)\in \mathbb{Q}^2$ satisfying equation \eqref{6.16}, such that $\eta(\alpha,\beta)$ represents $c$ on $T_{e_*}$. Then we can apply the above construction to 
$\alpha\tilde{\Lambda}+\beta\tilde{\Lambda}'$, to construct the desired surface $\Sigma_c=\Sigma_{\alpha,\beta}$. The desired conditions (i)-(iv) on $\Sigma_c$ follow from the above conditions (a)-(d) on $\Sigma_{\alpha,\beta}$ and the fact that $\tau:\tilde{M}\to \tilde{M}$ preserves the frame of $\tilde{M}$.

This finishes the proof of Proposition \ref{technicalresult}.
\end{proof}

\section{Graph $3$-manifold groups are Lex}\label{lexproof}

Now we can forget all the linear algebra technicalities in Sections \ref{moregraph} and \ref{technicalsection}, and we will only use the statement of Proposition \ref{technicalresult}. The goal of this section is to prove Theorem \ref{lex}, which is a consequence of the following Theorem \ref{lexsubgroups}. The proof of Theorem \ref{lexsubgroups} works for all chargeless closed graph $3$-manifolds.

\begin{theorem}\label{lexsubgroups}
Let $M$ be a closed graph $3$-manifold that is not virtually fibered, then $G=\pi_1(M)$ contains a sequence of subgroups 
$$G>G_1\vartriangleright G_2\vartriangleright  G_3\vartriangleright  G_4,$$ such that the following hold.
\begin{enumerate}
\item $G_1$ is a finite-index subgroup of $G$.
\item $G_1/G_2\cong \mathbb{Z}$.
\item The index of $G_3$ in $G_2$ is $1$ or $2$.
\item $G_3/G_4$ is isomorphic to a subgroup of $\mathbb{Q}$.
\item $G_4$ is isomorphic to a free group (possibly infinitely generated).
\end{enumerate}
\end{theorem}

We first prove a lemma. For any $2\times 2$ matrix $\theta=
\begin{pmatrix}
a & b\\
c &d 
\end{pmatrix}$ with rational entries, let $\hat{\theta}: \mathbb{Q}\cup \{\infty\}\to \mathbb{Q}\cup \{\infty\}$ be the bijection given by the fractional linear transformation corresponding to $\theta$, i.e. $\hat{\theta}(s)=\frac{as+b}{cs+d}$ for any $s\in \mathbb{Q}\cup \{\infty\}$.

\begin{lemma}\label{ai}
Let $\theta:\mathbb{Q}^2\to \mathbb{Q}^2$ be a linear isomorphism with determinant $1$, and let $\hat{\theta}:\mathbb{Q}\cup \{\infty\}\to \mathbb{Q}\cup \{\infty\}$ be the induced bijection. Then for any finite subset $S\subset \mathbb{Q}\cup \{\infty\}$, there exists $s_0\in \mathbb{Q}\cup \{\infty\}$, such that $\hat{\theta}^n(s_0)\notin S$ for all $n\in \mathbb{Z}$.
\end{lemma}

\begin{proof}

{\bf Case I.} If $|\text{tr}(\theta)|\geq 2$, then this lemma can be proved by a topological argument. 

The isomorphism $\theta$ gives an element in $\text{PSL}(2,\mathbb{Q})<\text{PSL}(2,\mathbb{R})\cong \text{Isom}_+(\mathbb{H}^2)$, and it acts on $\mathbb{Q}\cup \{\infty\}\subset \mathbb{R}\cup \{\infty\}=\partial \mathbb{H}^2$ via the fractional linear transformation $\hat{\theta}$. We also use $\hat{\theta}$ to denote the action on $\mathbb{H}^2$ and $\partial \mathbb{H}^2\cong S^1$.

Since $|\text{tr}(\theta)|\geq 2$, $\hat{\theta}$ is a parabolic or hyperbolic isometry of $\mathbb{H}^2$. Then there exist $s_{+\infty},s_{-\infty}\in \partial \mathbb{H}^2$ (possibly $s_{+\infty}=s_{-\infty}$), such that for any $x\in \partial \mathbb{H}^2$, we have 
$$\lim_{n\to +\infty}\hat{\theta}^n(x)=s_{+\infty},\ \lim_{n\to -\infty}\hat{\theta}^n(x)=s_{-\infty}.$$
Let $U$ be an open neighborhood of $\{s_{+\infty},s_{-\infty}\}\subset \partial \mathbb{H}^2$ such that $\partial \mathbb{H}^2\setminus U$ has non-empty interior. 

Since $S\subset \mathbb{Q}\cup \{\infty\}$ is a finite set, 
$$X=\{(n,s)\in \mathbb{Z}\times S\ |\ \hat{\theta}^n(s)\in \partial \mathbb{H}^2\setminus U\}$$ must be a finite set. Since $\partial \mathbb{H}^2\setminus U$ has non-empty interior and $\mathbb{Q}\cup \{\infty\}$ is dense in $\partial \mathbb{H}^2$, $(\mathbb{Q}\cup \{\infty\})\cap (\partial \mathbb{H}^2\setminus U)$ is an infinite set. So
there exists $s_0\in (\mathbb{Q}\cup \{\infty\})\cap (\partial \mathbb{H}^2\setminus U)$ such that $s_0\ne \hat{\theta}^n(s)$ for any $(n,s)\in X$. 

So  $s_0\ne \hat{\theta}^n(s)$ for any $n\in \mathbb{Z}$ and $s\in S$, thus $\hat{\theta}^n(s_0)\notin S$ for all $n\in \mathbb{Z}$.

{\bf Case II.} If $|\text{tr}(\theta)|< 2$, then $\hat{\theta}$ is an elliptic isometry of $\mathbb{H}^2$. The above argument still works if $\hat{\theta}$ is a finite order automorphism of $\mathbb{Q}\cup \{\infty\}$, but fails if $\hat{\theta}$ has infinite order. In the latter case, for any $x\in \partial \mathbb{H}^2$, the orbit of $x$ under the $\hat{\theta}$-action is dense in $\partial \mathbb{H}^2$. The following number-theoretic argument was generated by Google Gemini and checked by the author.

Let $\theta=\begin{pmatrix}
a & b\\
c & d
\end{pmatrix}$, and let $\Delta=(\text{tr}(\theta))^2-4\in \mathbb{Q}$. Then we have $\Delta<0$, and let $K=\mathbb{Q}(\sqrt{\Delta})$ be the corresponding quadratic number field with one complex place. Let $\alpha,\beta$ be the fixed points of $\hat{\theta}$ on the complex plane $\mathbb{C}$, then $\alpha,\beta\in K$ are complex conjugates of each other, $\text{Re}(\alpha)=\text{Re}(\beta)\in \mathbb{Q}$ and $0\ne \text{Im}(\alpha)=-\text{Im}(\beta)\in \mathbb{Q}\cdot \sqrt{-\Delta}$.

Let $T:\mathbb{Q}\cup \{\infty\}\to K^1=\{k\in K\ |\ |k|=1\}$ be defined by $T(s)=\frac{s-\alpha}{s-\beta}$. Here $T(s)\in K^1$ holds since $\alpha=\bar{\beta}$ and $\alpha,\beta\in K$. $T$ is the restriction of a map $\mathbb{R}\cup \{\infty\}\to S^1$ defined by the same formula. $T$ is clearly injective, since it is the restriction of a fractional linear transformation. $T$ is also surjective, since for any $k\in K^1$, we have $T^{-1}(k)=\frac{\alpha-\beta k}{1-k}$ and can check $T^{-1}(k)\in \mathbb{Q}\cup \{\infty\}$ as follows.

If $k=1$, then  $\frac{\alpha-\beta k}{1-k}=\infty \in \mathbb{Q}\cup \{\infty\}$. If $k\ne 1$, we have 
\begin{align*}
& \frac{\alpha-\beta k}{1-k}=\frac{\text{Re}(\alpha)(1-k)+i\cdot \text{Im}(\alpha)(1+k)}{1-k}=\text{Re}(\alpha)+i\cdot \text{Im}(\alpha)\frac{1+k-\bar{k}-k\bar{k}}{|1-k|^2}\\
=\ & \text{Re}(\alpha)+2i^2\cdot \text{Im}(\alpha)\text{Im}(k)\frac{1}{|1-k|^2}=\text{Re}(\alpha)-2\text{Im}(\alpha)\text{Im}(k)\frac{1}{(\text{Re}(1-k))^2+(\text{Im}(k))^2}\in \mathbb{Q}.
\end{align*}
Here, the third equation holds since $|k|=1$. The last item is rational since all the real parts lie in $\mathbb{Q}$, while all the imaginary parts lie in $\mathbb{Q}\cdot \sqrt{-\Delta}$ and they are multiplied together in pairs.

So the map $T:\mathbb{Q}\cup \{\infty\}\to K^1$ is a bijection. A direct computation gives
\begin{align*}
T(\hat{\theta}(s))=\frac{a-c\alpha}{a-c\beta}\cdot T(s)
\end{align*}
for all $s\in \mathbb{Q}\cup \{\infty\}$. Let $\xi=\frac{a-c\alpha}{a-c\beta}\in K^1$, then $T$ maps 
$$Y=\{\hat{\theta}^n(s)\ |\ n\in \mathbb{Z},s\in S\}\subset \mathbb{Q}\cup \{\infty\}\text{\ to\ } Z=\{\xi^n \cdot T(s)\ |\ n\in \mathbb{Z},s\in S\}\subset K^1,$$ and it suffices to prove that $K^1\setminus Z$ is not empty.
 
We consider $K^1$ as an abelian group under multiplication.
 Since $S$ is finite, $Z$ is a union of finitely many cosets of the cyclic subgroup $\langle \xi \rangle<K^1$ generated by $\xi$. To prove $K^1\setminus Z\ne \emptyset$, it suffices to prove that $K^1$ is not virtually cyclic.
 
We use $K^{\times}$ to denote $K\setminus \{0\}$, which is also an abelian group under multiplication. $\mathbb{Q}^{\times}$ is defined similarly and we have $\mathbb{Q}^{\times} <K^{\times}$. Let $U:K^{\times}\to K^1$ be the group homomorphism defined by $U(k)=k/\bar{k}$, then the kernel of $U$ is $\mathbb{Q}^{\times}<K^{\times}$. So $K^{\times}/\mathbb{Q}^{\times}$ is isomorphic to a subgroup of $K^1$. (Actually, Hilbert's Theorem 90 implies that $U$ is surjective and $K^1\cong K^{\times}/\mathbb{Q}^{\times}$.)
 
The main result in \cite{Bra} implies that $K^{\times}/\mathbb{Q}^{\times}$ is an infinitely generated abelian group. So $K^1$ is also an infinitely generated abelian group, which can not be virtually cyclic.

The proof of this lemma is done.
 \end{proof}

Now we are ready to prove Theorem \ref{lexsubgroups}.

\begin{proof}
{\bf Step I. Construction of $G_1<G$ with finite-index.} 

We first apply Lemma \ref{simplegraph} to construct a finite cover of $M$ with a frame structure, and we still denote this manifold by $M$.
Then we apply Proposition \ref{technicalresult} to construct a frame preserving finite cover $\tilde{M}\to M$, and denote $\tilde{M}$ by $M_1$. Then $G_1=\pi_1(M_1)$ is a finite-index subgroup of $G=\pi_1(M)$. The objects in $M_1$ and the subsurfaces $\Sigma_c$ given in Proposition \ref{technicalresult} will be used in the constructions below.
\bigskip

{\bf Step II. Construction of $G_2\vartriangleleft G_1$ such that $G_1/G_2\cong \mathbb{Z}$.} 

Let $T_{|e_*|},T_{|e_{**}|}\subset M_1$ be the JSJ tori given by Proposition \ref{technicalresult}, and let $N=M_1\setminus\setminus (T_{|e_*|}\cup T_{|e_{**}|})$. Let the boundary components of $N$ corresponding to $T_{|e_*|}$ be $T_{e_*}$ and $T_{\bar{e}_*}$, and let the boundary components of $N$ corresponding to $T_{|e_{**}|}$ be $T_{e_{**}}$ and $T_{\bar{e}_{**}}$. Proposition \ref{technicalresult} also gives an involution $\tau:N\to N$ that swaps $T_{e_*}$ with $T_{e_{**}}$, and swaps $T_{\bar{e}_*}$ with $T_{\bar{e}_{**}}$.

We take countably infinitely many copies of $N$ and denote them by $\{N_i\}_{i\in \mathbb{Z}}$, and denote the boundary components of $N_i$ by $T_{e_*,i}, T_{\bar{e}_*,i},T_{e_{**},i},T_{\bar{e}_{**},i}$, which correspond to $T_{e_*}, T_{\bar{e}_*},T_{e_{**}},T_{\bar{e}_{**}}$ in $\partial N$, repsectively. For any $i$, we paste $T_{e_*,i+1}\subset \partial N_{i+1}$ with $T_{\bar{e}_*,i}\subset \partial N_{i}$, and paste $T_{e_{**},i+1}\subset \partial N_{i+1}$ with $T_{\bar{e}_{**},i}\subset \partial N_{i}$, by using the pasting maps of $M_1$. This process gives a non-compact $3$-manifold $$M_2=\cup_{i\in \mathbb{Z}}N_i$$ that is an infinite cyclic cover of $M_1$. A picture of $M_2$ can be found in Figure \ref{Figure3}.

We take $G_2=\pi_1(M_2)<G_1=\pi_1(M_1)$. Then we have $G_2\vartriangleleft G_1$ and $G_1/G_2\cong \mathbb{Z}$.

\begin{figure}[htbp]
    \centering
    \def\svgwidth{5in} 
   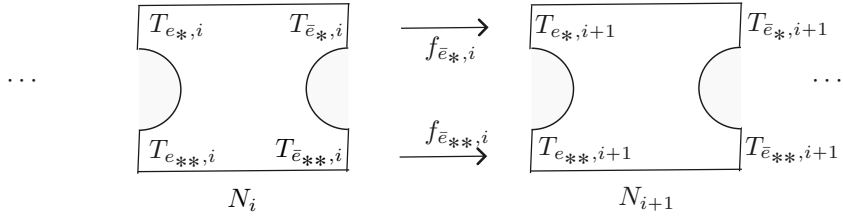
    \caption{A picture of the infinite cyclic cover $M_2$.}
     \label{Figure3}
\end{figure}

\bigskip

{\bf Step III. Construction of $G_3\vartriangleleft G_2$ with index $1$ or $2$.} 

We use $\theta$ to denote the isomorphism $(f_{\bar{e}_*})_{\#}\circ \iota:H_1(T_{e_*};\mathbb{Q})\to H_1(T_{e_*};\mathbb{Q})$ in Proposition \ref{technicalresult} (v), which has determinant $1$. We fix a basis of $H_1(T_{e_*};\mathbb{Q})\cong \mathbb{Q}^2$, let $s_1,s_2 \in \mathbb{Q}\cup \{\infty\}$ be the rational numbers corresponding to the projective slopes of $c_1^*,c_2^*$ in Proposition \ref{technicalresult}, and let $S=\{s_1,s_2\}$. Then Lemma \ref{ai} gives us an $s_0\in \mathbb{Q}\cup \{\infty\}$, such that $\hat{\theta}^i(s_0)\notin S$ for all $i\in \mathbb{Z}$. Here, $s_0$ corresponds to a projective slope on $T_{e_*}$, and let $c_0$ be an oriented rational slope representing it. Then for any $i\in \mathbb{Z}$, $c_i=\theta^i(c_0)$ is not a multiple of $c_1^*$ or $c_2^*$ for all $i\in \mathbb{Z}$. 

Now we apply Proposition \ref{technicalresult} to construct a connected, properly embedded, horizontal subsurface $\Sigma_i=\Sigma_{c_i}\to $$N_i$ whose boundaries are $K_i$ multiples of following unoriented rational slopes: $c_i$ on $T_{e_*,i}$, $\tau(c_i)$ on $T_{e_{**},i}$, $\iota(c_i)$ on $T_{\bar{e}_*,i}$, $\tau(\iota(c_i))$ on $T_{\bar{e}_{**},i}$, for some $K_i\in \mathbb{Z}\setminus \{0\}$.

Note that we have 
$$(f_{\bar{e}_*})_{\#}(\iota(c_i))=((f_{\bar{e}_*})_{\#}\circ \iota)(c_i)=\theta(c_i)=c_{i+1}$$ and 
$$(f_{\bar{e}_{**}})_{\#}\big(\tau(\iota(c_i))\big)=\tau\big(((f_{\bar{e}_*})_{\#}\circ \iota) (c_i)\big)=\tau(\theta(c_i))=\tau(c_{i+1}).$$ Here the first equality of the second equation holds since $\tau$ is induced by a homeormophism of $M_1$. 
Let $f_{\bar{e}_*,i}:T_{\bar{e}_*,i}\to T_{e_{*},i+1}$ and $f_{\bar{e}_{**},i}:T_{\bar{e}_{**},i}\to T_{e_{**},i+1}$ be the pasting maps between $N_{i}$ and $N_{i+1}$, as in Figure \ref{Figure3}. Then the above two equations imply that $f_{\bar{e}_*,i}(\Sigma_{i}\cap T_{\bar{e}_*,i})$ and $f_{\bar{e}_{**},i}(\Sigma_{i}\cap T_{\bar{e}_{**},i})$ are rational multiples of $\Sigma_{i+1}\cap T_{e_*,i+1}$ and $\Sigma_{i+1}\cap T_{e_{**},i+1}$, respectively, as unoriented multi-slopes. Moreover, the two multiplicative constants are the same.

In the following claim, we construct a double cover $(\tilde{N}_i,\tilde{\Sigma}_i)$ of $(N_i,\Sigma_i)$ to fix some orientability issue.

{\it Claim.} For any $i\in \mathbb{Z}$, there exists a (possibly disconnected) double cover $p_i:\tilde{N}_i\to N_i$ and a properly embedded, horizontal subsurface $\tilde{\Sigma}_i\to \tilde{N}_i$, such that the following hold.
\begin{enumerate}
\item[(i)] $\tilde{\Sigma}_i=p_i^{-1}(\Sigma_i)$ is orientable and we equip it with an orientation.
\item[(ii)] The boundary of $\tilde{N}_i$ consists of eight components, denoted by $\tilde{T}_{\bullet,i}^+,\tilde{T}_{\bullet,i}^-$, with $\bullet=e_*,e_{**},\bar{e}_*,\bar{e}_{**}$, such that $p_i$ maps both $\tilde{T}_{\bullet,i}^+$ and $\tilde{T}_{\bullet,i}^-$ to $T_{\bullet,i}$ by homeomorphisms.
\item[(iii)] We equip $\tilde{\Sigma}_i\cap \tilde{T}_{\bullet,i}^+$ and $\tilde{\Sigma}_i\cap \tilde{T}_{\bullet,i}^-$ with the boundary orientations induced by the orientation of $\tilde{\Sigma}_i$ (in item (i)), then they are mapped to opposite oriented multi-slopes on $T_{\bullet,i}$ (i.e. they give the same unoriented multi-slope, but have opposite orientations).
\end{enumerate}

If $\Sigma_i$ is orientable, we take $(\tilde{N}_i,\tilde{\Sigma}_i)$ to be two copies of $(N_i,\Sigma_i)$, denote the boundary components of one copy of $N_i$ by $\tilde{T}_{\bullet,i}^+$, and use $\tilde{T}_{\bullet,i}^-$ for the boundary components of the other copy of $N_i$. We equip two copies of $\Sigma_i$ with opposite orientations. Then $(\tilde{N}_i,\tilde{\Sigma}_i)$ clearly satisfies items (i)-(iii).

If $\Sigma_i$ is non-orientable, then its neignborhood $\mathcal{N}(\Sigma_i)\subset N_i$ is a twisted $I$-bundle over $\Sigma_i$, and $\partial \mathcal{N}(\Sigma_i)$ is a connected, orientable, horizontal subsurface in $N_i$. Since $\Sigma_i \subset N_i$ is horizontal, $N_i\setminus\setminus \mathcal{N}(\Sigma_i)$ is a twisted $I$-bundle over a non-orientable surface $\Sigma_i'\subset N_i\setminus\setminus \mathcal{N}(\Sigma_i)$. We take the surjective homomorphism 
$$\phi:\pi_1(N_i)\to \mathbb{Z}/2\mathbb{Z}$$ 
given by the mod-$2$ intersection number with $\Sigma_i\cup \Sigma_i'\subset N_i$.

Let $\tilde{N}_i$ be the double cover of $N_i$ corresponding to the kernel of $\phi$, and let $\sigma:\tilde{N}_i\to \tilde{N}_i$ be the non-trivial deck transformation. Let $\tilde{\Sigma}_i$ be the preimage of $\Sigma_i$ in $\tilde{N}_i$, then it is the orientable double cover of $\Sigma_i$. For any boundary component $T_{\bullet,i}\subset \partial N_i$, $\Sigma_i\cap T_{\bullet,i}$ is parallel to $\Sigma_i'\cap T_{\bullet,i}$ as unoriented multi-slopes, so the preimage of $T_{\bullet,i}$ in $\tilde{N}_i$ consists of two components: $\tilde{T}_{\bullet,i}^+$  and $\tilde{T}_{\bullet,i}^-$. We equip $\tilde{\Sigma}_i$ with an arbitrary orientation, then the deck transfomation $\sigma:\tilde{N}_i\to \tilde{N}_i$ maps $\tilde{\Sigma}_i$ to $-\tilde{\Sigma}_i$, as oriented surfaces. So $\sigma$ sends $\tilde{\Sigma}_i\cap \tilde{T}_{\bullet,i}^+$ to $(-\tilde{\Sigma}_i)\cap \tilde{T}_{\bullet,i}^-$, as oriented multi-slopes. Then item (iii) holds since $\sigma$ is a deck transformation of $\tilde{N}_i\to N_i$.

This finishes the proof of the Claim.

Since $\tilde{\Sigma}_i$ is an orientable, horizontal subsurface of $\tilde{N}_i$, $\tilde{N}_i\setminus \setminus \tilde{\Sigma}_i$ is a union of trivial $I$-bundles and twisted $I$-bundles (over non-orientable surfaces).
By Proposition \ref{technicalresult} (i), $\Sigma_i$ intersects with each vertex manifold of $N_i$ at a connected, orientable, horizontal subsurface. Then the same holds for $\tilde{\Sigma}_i\subset \tilde{N}_i$, thus for any vertex manifold $\tilde{N}_v\subset \tilde{N}_i$, $\tilde{N}_v\setminus \setminus (\tilde{\Sigma}_i\cap \tilde{N}_v)$ is connected. So $\tilde{N}_i\setminus \setminus \tilde{\Sigma}_i$ contains no twisted $I$-bundles, and $\tilde{\Sigma}_i\subset \tilde{N}_i$ is a fiber surface.

Then we paste $\{\tilde{N}_i\}_{i\in \mathbb{Z}}$ along their boundaries to obtain a double cover $M_3$ of $M_2$ such that the following hold.
\begin{enumerate}
\item[(a)] Each $\tilde{N}_i$ contains an embedded fiber surface $\tilde{\Sigma}_i$ such that items (i)-(iii) holds. 
\item [(b)] $\tilde{\Sigma}_i$ intersects with $\tilde{T}_{e_*,i}^{\pm},\tilde{T}_{e_{**},i}^{\pm}$ along $k_i$ copies of some oriented slope for the same $k_i\in \mathbb{Z}_{>0}$, and intersects with $\tilde{T}_{\bar{e}_*,i}^{\pm},\tilde{T}_{\bar{e}_{**},i}^{\pm}$ along $k_i'$ copies of some oriented slope for the same $k_i'\in \mathbb{Z}_{>0}$, and we may have $k_i\ne k_i'$.
\item [(c)] For any decomposition torus $T\subset M_2$ given by identifying $T_{\bar{\bullet},i}\subset \partial N_i$ and $T_{\bullet,i+1}\subset \partial N_{i+1}$, with $\bullet=e_*$ or $e_{**}$, we take a bijection between $\rho_{\bullet,i}:\{\tilde{T}_{\bar{\bullet},i}^+,\tilde{T}_{\bar{\bullet},i}^-\}\to \{\tilde{T}_{\bullet,i+1}^+,\tilde{T}_{\bullet,i+1}^-\}$ to paste these boundary tori of $\tilde{N}_i$ and $\tilde{N}_{i+1}$, by using the pasting map $T_{\bar{\bullet},i}\to T_{\bullet,i+1}$ in $M_2$. We claim that there exists a $\rho_{\bullet,i}$ so that for any $X\in \{\tilde{T}_{\bar{\bullet},i}^+,\tilde{T}_{\bar{\bullet},i}^-\}$, $\tilde{\Sigma}_i\cap X$ and $\tilde{\Sigma}_{i+1}\cap \rho_{\bullet,i}(X)$ are mapped to oriented multi-slopes in $T$ that are multiples of the same oriented slope on $T$, where the two multiplicative constants have opposite signs and may have different absolute values.
\end{enumerate}

Here, item (a) follows from the Claim. Item (b) follows from Proposition \ref{technicalresult} (ii)-(iv) and the construction of $(\tilde{N}_i,\tilde{\Sigma}_i)$. For item (c), the existence of $\rho_{\bullet,i}$ follows from item (iii) in the claim. Since $\tilde{\Sigma}_i\cap \tilde{T}_{\bar{\bullet},i}^+$ and $\tilde{\Sigma}_i\cap\tilde{T}_{\bar{\bullet},i}^-$ are mapped to opposite oriented multi-slopes on $T$, and so do $\tilde{\Sigma}_{i+1}\cap \tilde{T}_{\bullet,i+1}^+$ and $\tilde{\Sigma}_{i+1}\cap\tilde{T}_{\bullet,i+1}^-$, such a bijection $\rho_{\bullet,i}$ always exists.

The resulting manifold $M_3$ may have one or two components. If $M_3$ has two components, each component is a copy of $M_2$, and we abuse notation to use $M_3$ to denote one of its components. Then we take $G_3=\pi_1(M_3)<G_2=\pi_1(M_2)$, which is normal since $[G_2:G_3]=1$ or $2$.

\bigskip

{\bf Step IV. Construction of a free subgroup $G_4\vartriangleleft G_3$ such that $G_3/G_4$ is isomorphic to a subgroup of $\mathbb{Q}$.}

Now we apply the Mayer–Vietoris sequence to $M_3$ by taking
$$U=\bigcup_{i\text{\ even}}\tilde{N}_i\text{\ and\ }V=\bigcup_{i\text{\ odd}}\tilde{N}_i.$$
Then we have 
\begin{align}\label{7.1}
H_1(M_3;\mathbb{Z})\cong \Big( \oplus_{i\in \mathbb{Z}} H_1(\tilde{N}_i;\mathbb{Z})/\text{im}\big(\oplus_{i\in \mathbb{Z}}H_1(\tilde{N}_i\cap\tilde{N}_{i+1};\mathbb{Z})\big)\Big)\oplus \mathbb{Z}^{\infty}.
\end{align}
The $\mathbb{Z}^{\infty}$ in the second direct summand arises from
$\oplus H_0(\tilde{N}_i\cap \tilde{N}_{i+1};\mathbb{Z})$, since $\tilde{N}_i\cap \tilde{N}_{i+1}$ has more number of components than both $\tilde{N}_i$ and $\tilde{N}_{i+1}$ ($4$ versus $2$ or $1$ if $[G_2:G_3]=2$, and $2$ versus $1$ if $G_3=G_2$).
In the first direct summand, each $H_1(\tilde{N}_i\cap  \tilde{N}_{i+1};\mathbb{Z})$ is sent to $H_1(\tilde{N}_i;\mathbb{Z})\oplus H_1(\tilde{N}_{i+1};\mathbb{Z})$, where $H_1(\tilde{N}_i\cap \tilde{N}_{i+1};\mathbb{Z})\to H_1(\tilde{N}_i;\mathbb{Z})$ is induced by the inclusion, and $H_1(\tilde{N}_i\cap \tilde{N}_{i+1};\mathbb{Z})\to H_1(\tilde{N}_{i+1};\mathbb{Z})$ is the negative of the inclusion-induced homomorphism. 

Let $\phi_i:H_1(\tilde{N}_i;\mathbb{Z})\to \mathbb{Z}\cong H_1(S^1;\mathbb{Z})$ be the homomorphism induced by the fibration $\tilde{N}_i\to S^1$ dual to $\tilde{\Sigma}_i$, such that curves with positive algebraic intersection number with $\tilde{\Sigma}_i$ are sent to positive integers. We consider two homomorphisms 
\begin{align}\label{7.2}
\phi_{i,i+1}^{i}:H_1(\tilde{N}_i\cap \tilde{N}_{i+1};\mathbb{Z})\to H_1(\tilde{N}_i;\mathbb{Z})\xrightarrow{\phi_i}\mathbb{Z}
\end{align} and
\begin{align}\label{7.3}
\phi_{i,i+1}^{i+1}:H_1(\tilde{N}_i\cap \tilde{N}_{i+1};\mathbb{Z})\to H_1(\tilde{N}_{i+1};\mathbb{Z})\xrightarrow{\phi_{i+1}}\mathbb{Z},
\end{align} 
where both first homomorphisms are induced by inclusions. By items (b) and (c) in Step III, as oriented multi-slopes on $\tilde{N}_i\cap \tilde{N}_{i+1}$ (consisting of two or four tori), we have 
$$k_{i+1}\cdot (\tilde{\Sigma}_i\cap (\tilde{N}_i\cap \tilde{N}_{i+1}))=-k_i'\cdot (\tilde{\Sigma}_{i+1}\cap (\tilde{N}_i\cap \tilde{N}_{i+1})).$$
This implies that 
\begin{align}\label{7.4}
k_i'\cdot \phi_{i,i+1}^{i+1}=k_{i+1}\cdot \phi_{i,i+1}^{i}.
\end{align}
Here the negative sign disappears since $\tilde{N}_i\cap \tilde{N}_{i+1}$ has opposite orientations induced from $\tilde{N}_i$ and $\tilde{N}_{i+1}$.

We let $\psi_i:H_1(\tilde{N}_i;\mathbb{Z})\to \mathbb{\mathbb{Q}}$ be defined by 
\begin{align}\label{7.5}
\psi_i(x)=
\begin{cases}
(\Pi_{j=1}^i \frac{k_{j-1}'}{k_j})\cdot \phi_i(x) & \text{if\ } i\geq 0,\\
(\Pi_{j=0}^{-i-1} \frac{k_{-j}}{k_{-j-1}'})\cdot \phi_i(x) & \text{if\ } i<0.
\end{cases}
\end{align}
We define $\psi_{i,i+1}^i,\psi_{i,i+1}^{i+1}:H_1(\tilde{N}_i\cap \tilde{N}_{i+1};\mathbb{Z})\to \mathbb{Q}$ as in equations \eqref{7.2} and \eqref{7.3}, by replacing $\phi_i$ and $\phi_{i+1}$ by $\psi_i$ and $\psi_{i+1}$, respectively. Then equation \eqref{7.4} implies that
\begin{align}\label{7.6}
\psi_{i,i+1}^i=\psi_{i,i+1}^{i+1}.
\end{align}

Now we construct a homomorphism 
$$\psi':\Big(\oplus_{i\in \mathbb{Z}}H_1(\tilde{N}_i;\mathbb{Z})\Big)\oplus \mathbb{Z}^{\infty}\to \mathbb{Q}$$
that restricts to each $H_1(\tilde{N}_i;\mathbb{Z})$ by $\psi_i$ and restricts to $\mathbb{Z}^{\infty}$ by $0$. Then equation \eqref{7.6} implies that $\psi'$ vanishes on $\text{im}\big(\oplus_{i\in \mathbb{Z}}H_1(\tilde{N}_i\cap  \tilde{N}_{i+1};\mathbb{Z})\big)<\oplus_{i\in \mathbb{Z}}H_1(\tilde{N}_i;\mathbb{Z})$, and equation \eqref{7.1} implies that $\psi'$ induces a well-defined homomorphism
\begin{align}\label{7.7}
\psi: H_1(M_3;\mathbb{Z})\to \mathbb{Q}.
\end{align}
Note that $\text{im}(\psi)$ contains $\mathbb{Z}\subset \mathbb{Q}$, by restricting $\psi$ to $H_1(\tilde{N}_0;\mathbb{Z})$. The image of $\psi$ depends on the sequence of positive integers $\{k_i,k_i'\}_{i\in \mathbb{Z}}$, so it may not be finitely generated, and may not be equal to $\mathbb{Q}$.

Let $p_4:M_4\to M_3$ be the covering space of $M_3$ corresponding to 
$$\text{ker}(\pi_1(M_3)\to H_1(M_3;\mathbb{Z})\xrightarrow{\psi}\mathbb{Q})<G_3=\pi_1(M_3),$$
and let $G_4=\pi_1(M_4)$. Then we have $G_4\vartriangleleft G_3$ and $G_3/G_4$ is isomorphic to a subgroup of $\mathbb{Q}$. It remains to check that $G_4=\pi_1(M_4)$ is a free group.

For any $\tilde{N}_i\subset M_3$, the composition 
$$\pi_1(\tilde{N}_i)\to \pi_1(M_3)\to H_1(M_3;\mathbb{Z})\xrightarrow{\psi}\mathbb{Q}$$ is given by the dual of the fiber surface $\tilde{\Sigma}_i$, up to multiplying a non-zero rational number. So each component of $p_4^{-1}(\tilde{N}_i)\subset M_4$ is homeomorphic to $\tilde{\Sigma}_i\times \mathbb{R}$, which is homotopy equivalent to $\tilde{\Sigma}_i$, and the number of components of $p_4^{-1}(\tilde{N}_i)$ is finite or countably infinite.

So up to homotopy equivalence, $M_4$ can be obtained by pasting infinitely many compact surfaces (copies of $\tilde{\Sigma}_i$) along their boundaries. Therefore, $M_4$ is homotopy equivalent to a non-compact surface, and $\pi_1(M_4)$ is isomorphic to an (infinitely generated) free group.

This finishes the proof of Proposition \ref{lexsubgroups}.
\end{proof}

Now we are ready to prove Theorem \ref{lex}, which states that groups of all non-virtually fibered closed graph $3$-manifolds lie in the family Lex.

\begin{proof}



Let $M$ be a closed graph $3$-manifold that is not virtually fibered.
Then Proposition \ref{lexsubgroups} provides us a sequence of subgroups 
$$\pi_1(M)=G>G_1\vartriangleright G_2\vartriangleright G_3\vartriangleright G_4$$
such that $[G:G_1],[G_2:G_3]$ are finite, $G_1/G_2,G_3/G_4$ are abelian, and $G_4$ is free.

Since $G_4$ is free, Proposition \ref{lexproperty} (1) implies that $G_4$ lies in the family Lex. Since abelian groups are amenable and $G_3/G_4$ is abelian, Proposition \ref{lexproperty} (3) implies that $G_3$ lies in the family Lex. Then Proposition \ref{lexproperty} (2) implies that $G_2$ lies in the family Lex, Proposition \ref{lexproperty} (3) implies that $G_1$ lies in the family Lex (since $G_1/G_2\cong \mathbb{Z}$), and Proposition \ref{lexproperty} (2) implies that $G$ lies in the family Lex.

\end{proof}

\section{All finitely generated $3$-manifold groups lie in the family Lex}\label{morelexsection}
Proposition \ref{lexsubgroups} inspires the following definition.
\begin{definition}\label{starn}
For any integer $n\geq 0$, we say that a group $G$ satisfies {\it property $(\star_n)$} if it contains a sequence of subgroups 
\begin{align}\label{8.1}
G=G_0\vartriangleright  G_1\vartriangleright  G_2\vartriangleright  \cdots \vartriangleright  G_{2n}
\end{align}
such that the following hold.
\begin{enumerate}
\item For any $i=0,\cdots,n-1$, $G_{2i}/G_{2i+1}$ is a finite group.
\item For any $i=0,\cdots,n-1$, $G_{2i+1}/G_{2i+2}$ is an abelian group (possibly infinitely generated).
\item $G_{2n}$ is isomorphic to a free group (possibly infinitely generated).
\end{enumerate}
\end{definition}
Here, we allow the finite, abelian, and free groups in the definition to be trivial. So if a group satisfies property $(\star_n)$, then it satisfies property $(\star_m)$ for any $m\geq n$.

The proof of Theorem \ref{lex} (at the end of the last section) implies that any group $G$ satisfying property $(\star_n)$ for some $n$ lies in the family Lex.

\begin{lemma}\label{elementarylex}
If $G$ is a group satisfying property $(\star_n)$ for some integer $n\geq 0$, then $G$ lies in the family Lex.
\end{lemma}

Here are some elementary properties of groups satisfying property $(\star_n)$.
\begin{lemma}\label{starnproperty}
Let $G,H$ be two groups satisfying property $(\star_n)$ for some integer $n\geq 0$.
\begin{enumerate}
\item If $A<G$ is a subgroup, then $A$ satisfies property $(\star_n)$.
\item If $G<B$ is a finite-index subgroup, then $B$ satisfies property $(\star_n)$.
\item $G*H$ satisfies property $(\star_n)$.
\end{enumerate}
\end{lemma}

\begin{proof}
We first prove item (1). Let 
$$G=G_0\vartriangleright G_1\vartriangleright G_2\vartriangleright \cdots \vartriangleright G_{2n}$$
be a sequence of subgroups satisfying Definition \ref{starn}. Then we take $A_i=A\cap G_i$ to get a sequence of subgroups
\begin{align}\label{8.2}
A=A_0>A_1>A_2>\cdots>A_{2n}.
\end{align}
Since $G_{i+1}\vartriangleleft  G_i$, it is clear that $A_{i+1}=A\cap G_{i+1}$ is a normal subgroup of $A_i=A\cap G_i$.
Moreover, we have
$$A_i/A_{i+1}=A\cap G_i/A\cap G_{i+1}\cong (A\cap G_i)\cdot G_{i+1}/G_{i+1}<G_i/G_{i+1}.$$
Since a subgroup of a finite/abelian group is still finite/abelian, $A_i/A_{i+1}$ satisfies the condition in Definition \ref{starn} (1) (2). Since $A_{2n}=A\cap G_{2n}$ is a subgroup of a free group $G_{2n}$, $A_{2n}$ is also a free group. So the sequence of subgroups in \eqref{8.2} satisfies the condition in Definition \ref{starn}, thus $A$ has property $(\star_n)$.

Now we prove item (2). Since $G_1<G$ and $G<B$ are both finite-index subgroups, there exists a finite-index normal subgroup $B_1\vartriangleleft B$ contained in $G_1$. Then we take $B_i=G_i\cap B_1$ for $i=2,\cdots,2n$, and we use 
$$B=B_0\vartriangleright B_1>B_2>\cdots>B_{2n}$$ 
to check that $B$ satisfies property $(\star_n)$, as in the proof of item (1). 

We prove item (3) by inductively proving the following stronger statement. For any integer $n\geq 0$, let $$G=*_{\lambda\in \Lambda}G_{\lambda}$$
be a free product such that $\{G_{\lambda}\}_{\lambda\in \Lambda}$ has only finitely many isomorphic types and each $G_{\lambda}$ satisfies property $(\star_n)$, then $G$ satisfies property $(\star_n)$. Here we have no restriction on the cardinality of $\Lambda$, but we will only apply the result to countable $\Lambda$, since we only care about finitely generated $3$-manifold groups.

If $n=0$, then definition of property $(\star_0)$ implies that each $G_{\lambda}$ is isomorphic to a free group, then $G=*_{\lambda\in \Lambda}G_{\lambda}$ is also isomorphic to a free group and satisfies property $(\star_0)$.

Now, suppose the above statement holds for $n$, and we want to prove it for $n+1$. Suppose that $G=*_{\lambda\in \Lambda}G_{\lambda}$ such that each $G_{\lambda}$ is isomorphic to one of $G_{\lambda_1},\cdots,G_{\lambda_k}$, and each $G_{\lambda_i}$ satisfies property $(\star_{n+1})$. For any $i=1,\cdots,k$, $G_{\lambda_i}$ has a sequence of subgroups satisfying Definition \ref{starn}:
$$G_{\lambda_i}=G_{\lambda_i,0}\vartriangleright G_{\lambda_i,1}\vartriangleright G_{\lambda_i,2}\vartriangleright \cdots \vartriangleright G_{\lambda_i,2n+2}.$$

We take the surjective homomorphism 
$$\phi:G=*_{\lambda\in \Lambda}G_{\lambda}\to \Pi_{i=1}^k G_{\lambda_i}/G_{\lambda_i,1}$$
that restricts to each free factor $G_{\lambda}$ as 
$$G_{\lambda}\to G_{\lambda_i}\to G_{\lambda_i}/G_{\lambda_i,1}\to \Pi_{i=1}^k G_{\lambda_i}/G_{\lambda_i,1}.$$
Here, the first homomorphism is an isomorphism to some $G_{\lambda_i}$ with $i=1,\cdots,k$, the second homomorphism is the quotient homomorphism, and the third one is the natural inclusion. Then we take $G_1=\text{ker}(\phi)$, which is a finite-index normal subgroup of $G$ since $\Pi_{i=1}^k G_{\lambda_i}/G_{\lambda_i,1}$ is a finite group. 

The Kurosh Subgroup Theorem implies that 
$$G_1\cong (*_{\lambda'\in \Lambda_1}G_{\lambda',1})*F$$ is a free product, where each $G_{\lambda',1}$ is isomorphic to one of $G_{\lambda_1,1},\cdots, G_{\lambda_k,1}$ and $F$ is a free group.
Then we take the surjective homomorphism 
$$\psi:G_1= (*_{\lambda'\in \Lambda_1}G_{\lambda',1})*F\to \Pi_{i=1}^k G_{\lambda_i,1}/G_{\lambda_i,2}$$
that restricts to $F$ trivially and restricts to each free factor $G_{\lambda',1}$ as 
$$G_{\lambda',1}\to G_{\lambda_i,1}\to G_{\lambda_i,1}/G_{\lambda_i,2}\to \Pi_{i=1}^k G_{\lambda_i,1}/G_{\lambda_i,2}.$$
Here, the first homomorphism is an isomorphism to some $G_{\lambda_i,1}$ with $i=1,\cdots,k$, the second and third are quotient and inclusion homomorphisms, respectively. Then we take $G_2=\text{ker}(\psi)\vartriangleleft G_1$, and $G_1/G_2\cong\Pi_{i=1}^k G_{\lambda_i,1}/G_{\lambda_i,2}$ is abelian. 

By the Kurosh Subgroup Theorem again, we have $$G_2\cong (*_{\lambda''\in \Lambda_2}G_{\lambda'',2})*F'$$
where each $G_{\lambda'',2}$ is isomorphic to one of $G_{\lambda_1,2},\cdots, G_{\lambda_k,2}$ and $F'$ is a free group. 

Since $G_{\lambda_1,2},\cdots, G_{\lambda_k,2}$ and $F'$ all satisfy property $(\star_n)$, the inductive hypothesis implies that $G_2$ satisfies property $(\star_n)$. So $G$ satisfies property $(\star_{n+1})$ and the induction is done.
\end{proof}

Note that the sequence of subgroups in Proposition \ref{lexsubgroups} may not satisfy the requirement in Definition \ref{starn}, since the $G_1<G$ in Proposition \ref{lexsubgroups} may not be a normal subgroup. So we first prove the following lemma, which implies that groups of $3$-manifolds in Theorem \ref{virtualfiber} satisfy property $(\star_2)$, thus lie in the family Lex.
\begin{lemma}\label{3manifoldstar2}
Let $M$ be a compact, connected, orientable, irreducible $3$-manifold with empty or tori boundary, then $G=\pi_1(M)$ satsifies property $(\star_2)$.
\end{lemma}
\begin{proof}
If $G=\pi_1(M)$ is finite, then we take $G_1=G_2=\{1\}$. Then the sequence $G\vartriangleright G_1\vartriangleright G_2$ implies that $G$ satisfies property $(\star_1)$, thus also satisfies property $(\star_2)$.

If $M$ is virtually a surface bundle over the circle, let $\tilde{M}$ be a finite regular cover of $M$ that is a surface bundle over the circle, with fiber surface $\Sigma$. We construct
$$G=G_0\vartriangleright G_1\vartriangleright G_2\vartriangleright G_3\vartriangleright G_4$$ 
as follows. We first take $G_1=\pi_1(\tilde{M}) \vartriangleleft G_0=\pi_1(M)$ of finite-index. Let $\phi: G_1\to \mathbb{Z}$ be the surjective homomorphism dual to $\Sigma$. We take $G_2=G_3=\text{ker}(\phi)\cong \pi_1(\Sigma)$, then we have $G_1/G_2\cong \mathbb{Z}$ and $G_2/G_3\cong \{1\}$. Here, $G_3$ is the trivial group, a free group, or a surface group. If $G_3$ is trivial, we take $G_4=G_3$. If $G_3$ is nontrivial, there exists a surjective homomorphism $\psi: G_3 \to \mathbb{Z}$ and $G_4=\text{ker}(\psi)$ is isomorphic to a free group. So $G=\pi_1(M)$ satisfies property $(\star_2)$ in this case.

If $M$ is virtually a circle bundle over a surface, let $\tilde{M}$ be a finite regular cover of $M$ that is a surface bundle over the circle, with an orientable base surface $\Sigma$. If $\Sigma$ is the $2$-disk, we have $G_1=\pi_1(\tilde{M})\cong \mathbb{Z}$, so $G$ satisfies property $(\star_1)$, thus also property $(\star_2)$. So we can assume that $\pi_1(\Sigma)$ is either a free group or a surface group. We construct
$$G=G_0\vartriangleright G_1\vartriangleright G_2\vartriangleright G_3\vartriangleright G_4$$ 
as follows.
We take $G_1=\pi_1(\tilde{M}) \vartriangleleft G=\pi_1(M)$ of finite-index. Let $\phi:\pi_1(\Sigma)\to \mathbb{Z}$ be a surjective homomorphism, and let $\tilde{\Sigma}$ be the infinite cyclic cover of $\Sigma$ corresponding to $\text{ker}(\phi)$. We take the homomorphism $\psi:\pi_1(\tilde{M})\to \pi_1(\Sigma)\xrightarrow{\phi}\mathbb{Z}$ where the first homomorphism is induced by the bundle projection, and let $\hat{M}$ be the infinite cyclic cover of $\tilde{M}$ corresponding to $\text{ker}(\psi)$. We take $G_2=G_3=\pi_1(\hat{M})$,  then $G_1/G_2\cong \mathbb{Z}$ and $G_2/G_3\cong \{1\}$ hold. Since $\hat{M}$ is a circle bundle over the non-compact orientable surface $\tilde{\Sigma}$, it must be the trivial bundle and we have $G_3=\pi_1(\hat{M})\cong \pi_1(\tilde{\Sigma})\times \pi_1(S^1)\cong \pi_1(\tilde{\Sigma})\times \mathbb{Z}$. Then we take $G_4=\pi_1(\tilde{\Sigma})\times \{1\} \vartriangleleft G_3$, which is a free group since $\tilde{\Sigma}$ is non-compact. So $G=\pi_1(M)$ satisfies property $(\star_2)$ in this case.

By Theorem \ref{virtualfiber}, the remaining case is that $M$ is a closed graph $3$-manifold that is not virtually fibered. Proposition \ref{lexsubgroups} gives a sequence of subgroups 
$$G=G_0>G_1\vartriangleright G_2\vartriangleright G_3\vartriangleright G_4$$ 
satisfying the definition of property $(\star_2)$, except that $G_1$ may not be a normal subgroup of $G$. Let $H_1\vartriangleleft G$ be a finite-index normal subgroup contained in $G_1$, and let $H_i=H_1\cap G_i$ for $i=2,3,4$. Then the sequence of subgroups $$G=G_0\vartriangleright H_1>H_2>H_3>H_4$$ 
satisfyies the requirement of property $(\star_2)$. We can check these properties by the same proof as in Lemma \ref{starnproperty} (1). 

The proof of this lemma is done.
\end{proof}

Now we are ready to prove Corollary \ref{morelex}, which implies that all finitely generated $3$-manifold groups lie in the family Lex. The structure of the proof is similar to that of Corollary \ref{morepolyfree}.
\begin{proof}

We will prove that all finitely generated $3$-manifold groups satisfy property $(\star_2)$. Then Lemma \ref{elementarylex} implies that all these groups lie in the family Lex.

{\bf Step I.} We first suppose that $M$ is a compact, orientable, irreducible $3$-manifold with empty or tori boundary. Then Lemma \ref{3manifoldstar2} implies that $G=\pi_1(M)$ satisfies property $(\star_2)$.

{\bf Step II.} Now we suppose that $M$ is compact, orientable, irreducible and $\partial$-irreducible. 

We can assume that $M$ has no $S^2$ boundary component. Otherwise, the irreducibility implies $M=D^3$ and $\pi_1(M)$ is trivial.

As in the proof of Corollary \ref{morepolyfree}, we construct a $3$-manifold $N$ as in Step I, so that $M$ is a submanifold of $N$ and the inclusion $M\to N$ induces an injective homomorphism $\pi_1(M)\to \pi_1(N)$. Step I implies that $\pi_1(N)$ satisfies property $(\star_2)$, then Lemma \ref{starnproperty} (1) implies that $G=\pi_1(M)$ satisfies property $(\star_2)$.

{\bf Step III.} Then we suppose that $M$ is compact, orientable, and irreducible. 

We inductively compress $\partial M$ along properly embedded discs to obtain a disjoint union of manifolds as in Step II. So $\pi_1(M)$ is a (finite) free product of groups satisfying property $(\star_2)$ and a free group. Then Lemma \ref{starnproperty} (3) implies that $\pi_1(M)$ satisfies property $(\star_2)$.

{\bf Step IV.} We suppose that $M$ is compact and orientable.

The prime decomposition implies that $M$ is a connected sum of manifolds as in Step IV and copies of $S^2\times S^1$. So $\pi_1(M)$ is a (finite) free product of groups satisfying property $(\star_2)$ and a free group, and it satisfies property $(\star_2)$ by Lemma \ref{starnproperty} (3) again.

{\bf Step V.} We suppose that $M$ is an arbitrary $3$-manifold with finitely generated fundamental group. 

If $M$ is non-orientable, we take the orientable double cover $\tilde{M}$. If $\tilde{M}$ is not compact, we take a Scott core $N\subset \tilde{M}$ (\cite{Sco}), which is a compact manifold as in Step IV such that the inclusion-induced homomorphism $\pi_1(N)\to \pi_1(\tilde{M})$ is an isomorphism. By Step IV, $\pi_1(N)\cong \pi_1(\tilde{M})$ satisfies property $(\star_2)$, then Lemma \ref{starnproperty} (2) implies that $\pi_1(M)$ satisfies property $(\star_2)$.

We finish the proof of Corollary \ref{morelex}.
\end{proof}

\end{document}